\documentclass[12pt]{article}
\usepackage{amsmath,amssymb, latexsym}
\usepackage{psfrag}
\usepackage{color}
\title{Diophantine tori and Weyl laws for non-selfadjoint operators in dimension two}
\author{Michael Hitrik\\Department of Mathematics \\University of California \\ Los Angeles
\\CA 90095-1555, USA\\hitrik@math.ucla.edu \and
Johannes Sj\"ostrand\\IMB, Universit\'e de Bourgogne\\9, Av. A. Savary, BP 47870\\FR-21078 Dijon C\'edex France\\and UMR 5584 CNRS
\\johannes.sjostrand@u-bourgogne.fr}
\date{}
\newcommand{\real}{\mbox{\bf R}}

\def\abs#1{\left|#1\right|}
\def\begeq{\begin{equation}}
\def\endeq{\end{equation}}
\renewcommand{\Im}{\mbox{\rm Im\,}}

\def\wrtext#1{\relax\ifmmode{\leavevmode\hbox{#1}}\else{#1}\fi}

\newcommand{\eps}{\varepsilon}
\def\part#1{\frac{\partial}{\partial #1}}

\def\norm#1{||\,#1\,||}

\newcommand{\comp}{\mbox{\bf C}}
\newcommand{\z}{\mbox{\bf Z}}
\newcommand{\nat}{\mbox{\bf N}}

\renewcommand{\Re}{\mbox{\rm Re\,}}
\renewcommand{\exp}{\mbox{\rm exp\,}}
\newcommand{\supp}{\mbox{\rm supp}}

\newcommand{\neigh}{neighborhood}

\newtheorem{dref}{Definition}[section]
\newtheorem{lemma}[dref]{Lemma}
\newtheorem{theo}[dref]{Theorem}
\newtheorem{prop}[dref]{Proposition}

\newenvironment{proof}{\vspace{.3cm}\noindent{{\em Proof:}}}{\hfill$\Box$
\vspace{.2cm}}

\begin{document}
\maketitle

\vspace*{1cm}
\noindent
{\bf Abstract}: We study the distribution of eigenvalues for non-selfadjoint perturbations of selfadjoint semiclassical
analytic pseudodifferential operators in dimension two, assuming that the classical flow of the unperturbed part is completely integrable.
An asymptotic formula of Weyl type for the number of eigenvalues in a spectral band, bounded from above and from below by levels corresponding
to Diophantine invariant Lagrangian tori, is established. The Weyl law is given in terms of the long time averages
of the leading non-selfadjoint perturbation along the classical flow of the unperturbed part.



\tableofcontents
\section{Introduction}
\label{intro}
\setcounter{equation}{0}
Among the various results in spectral analysis for selfadjoint linear partial differential operators, a distinguished role is played by the
Weyl law, remarkable for its generality and simplicity--- see~\cite{DiSj},~\cite{Ivrii} for a detailed discussion and numerous references. Let us
recall here a rough statement of it in the semiclassical case. Let $P = p^w(x,hD_x)$, $0<h \leq 1$, be the semiclassical Weyl
quantization on $\real^n$ of a real-valued smooth symbol $p$ belonging to a suitable symbol class and satisfying an ellipticity condition at infinity,
guaranteeing that the spectrum of $P$ is discrete in a fixed open set $\Omega \subset \real$. When $[E_1,E_2] \subset \Omega$ is an
interval such that ${\rm vol}(p^{-1}(E_j))=0$, $j=1,2$, then the number $N(P,[E_1,E_2])$ of eigenvalues of $P$ in $[E_1,E_2]$ satisfies
\begeq
\label{Weyl0}
N(P,[E_1,E_2])  = \frac{1}{(2\pi h)^n} \left({\rm vol}(p^{-1}([E_1,E_2])) + o(1)\right),\quad h\rightarrow 0.
\endeq
The leading coefficient in the Weyl law (\ref{Weyl0}), given by the phase space volume corresponding to the energy range $[E_1,E_2]$,
captures the basic physical intuition of each quantum state occupying a fixed volume $(2\pi h)^n$ in the phase space. We remark that the
corresponding development for selfadjoint partial differential operators in the high energy limit has a long and distinguished tradition, starting
with the works of H. Weyl, see~\cite{Weyl}.

\medskip
\noindent
The situation becomes quite different in the non-selfadjoint analytic case, where the asymptotics of the counting function for eigenvalues may no
longer be governed by volumes of subsets of the real phase space, and the spectrum is often determined by the behavior of the holomorphic continuation
of the symbol along suitable complex deformations of the real phase space. Following~\cite{Davies}, \cite{HSj}, we may consider, for instance, the
complex harmonic oscillator $P = 1/2((hD_x)^2 + ix^2)$ on $L^2(\real)$. Here the range of the symbol on $\real^2$ is the closed first quadrant, while
according to the general results of~\cite{Sj74}, the spectrum of $P$ is equal to $\{e^{i\pi /4} (k+\frac{1}{2})h; k\in \nat\}$. In this case, it turns
out that the natural phase space associated to $P$ is given by $T^*(e^{-i\frac{\pi}{8}}\real)$ --- see also~\cite{Hi1} for the precise spectral results
in dimension one, closely related to this example.

\medskip
\noindent
In the paper~\cite{MeSj2}, it has been shown that for large and stable classes of non-selfadjoint analytic operators
in dimension two, the individual eigenvalues can be determined accurately in the semiclassical limit by means of a Bohr-Sommerfeld quantization condition, defined in terms of suitable complex Lagrangian tori close to the real domain.
(See also~\cite{Sj08} for the formulation of the corresponding Weyl laws.) The work~\cite{MeSj2} was subsequently continued in a series of papers
\cite{HiSj04}--\cite{HiSj3a}, \cite{HiSjVu},\cite{HiSj08}, all of them concerned with the case of non-selfadjoint perturbations of selfadjoint operators
of the form,
$$
P_{\eps}(x,hD_x)  = p^w(x,hD_x) + i\eps q^w(x,hD_x),\quad 0< \eps \ll 1,
$$
with the leading symbol $p_{\eps}(x,\xi) = p(x,\xi) + i\eps q(x,\xi)$, $(x,\xi)\in T^*\real^2$. Here $p$ is real, so that $P_{\eps=0}$ is selfadjoint, and
we assume that both $p$ and $q$ are analytic, with $p$ elliptic outside a compact set. The spectrum of $P_{\eps}$ near the origin is confined to a band
of height ${\cal O}(\eps)$, and to study the imaginary parts of the eigenvalues in the band, following the classical averaging
method~\cite{Weinstein77},~\cite{Guillemin81}, we introduce the time averages
\begeq
\label{q_av_0}
\langle{q}\rangle_T  = \frac{1}{T}\int_0^T q\circ \exp(tH_p)\,dt,\quad T>0,
\endeq
of $q$ along the classical trajectories of $p$.

\medskip
\noindent
The main focus of the works~\cite{HiSjVu} and \cite{HiSj08} was on the case when the $H_p$--flow is completely integrable. The real energy surface
$p^{-1}(0)$ is then foliated by invariant Lagrangian tori, along with possibly some other more complicated invariant sets. Given an invariant torus
$\Lambda\subset p^{-1}(0)$, which is Diophantine (i.e. the rotation number of the $H_p$--flow along $\Lambda$ is poorly approximated by rational numbers), or more generally,
irrational, then the time averages $\langle{q}\rangle_T$ along $\Lambda$ converge to the space average $\langle{q}\rangle(\Lambda)$
of $q$ over $\Lambda$, as $T\to \infty $. When $\Lambda $ is a torus with a rational rotation number, or a more general singular invariant set in the
foliation of the energy surface $p^{-1}(0)$, then we need to consider the whole interval $Q_\infty (\Lambda)$ of limits of the flow averages above.

\medskip
\noindent
The principal result of~\cite{HiSjVu} says, somewhat roughly, that if $F_0\in \real$ is a value such that
$F_0=\langle{q}\rangle(\Lambda_0)$ for a single Diophantine torus $\Lambda_0 \subset p^{-1}(0)$, and $F_0$ does
not belong to $Q_\infty (\Lambda)$ for any other invariant set $\Lambda$ in
the energy surface, then the spectrum of $P_{\eps}$ can be completely determined in a rectangle $[-h^\delta /C,h^\delta /C]+i\eps
[F_0-h^\delta /C,F_0+h^\delta /C]$ modulo ${\cal O}(h^{\infty })$,
where $\delta$ is a positive exponent that can be chosen
arbitrarily small, and $\eps$ may vary in any interval of the form $h^K<\eps \ll 1$. The spectrum has a structure of a distorted lattice, with
horizontal spacing $\sim h$ and vertical spacing $\sim \eps h$. In the work~\cite{HiSj08}, we continued the analysis of the completely integrable case
by investigating what happens when the value $F_0$ belongs in addition to finitely many intervals $Q_{\infty}(\Lambda)$, corresponding to
rational invariant tori $\Lambda$. It was shown in~\cite{HiSj08} that the number of eigenvalues that can be created by such tori is much smaller than the number of eigenvalues
coming from the Diophantine ones, provided that the strength $\eps$ of the non-selfadjoint perturbation satisfies $h\ll \eps \leq h^{2/3+\delta}$,
$\delta>0$, and assuming that $F_0\in Q_{\infty}(\Lambda)\backslash \langle{q}\rangle(\Lambda)$, for each rational torus $\Lambda$.

\medskip
\noindent
The purpose of the present paper is to investigate the global distribution of the imaginary parts of eigenvalues of $P_{\eps}$ in an intermediate
spectral band, bounded from above and from below by levels such as $F_0$, described above. In fact,
when doing so, we shall refrain from treating the more general configurations considered in~\cite{HiSj08}, with both Diophantine and rational tori
present, and shall concentrate instead on the simpler situation of~\cite{HiSjVu}, where only Diophantine invariant tori
corresponding to $F_0$, occur. That such a study is planned by the authors was mentioned already in the introduction of~\cite{HiSjVu}, and here we are finally able
to present the result, giving a Weyl type asymptotic formula for the number of eigenvalues of $P_{\eps}$ in such a spectral band. Roughly speaking,
the main result of the present paper is as follows: let $F_j\in \real$, $j=1,3$, $F_3 < F_1$, be such that $F_j=\langle{q}\rangle(\Lambda_j)$,
where $\Lambda_j \subset p^{-1}(0)$ are Diophantine tori, and assume that $F_j$ does not belong to $Q_{\infty}(\Lambda)$ when $\Lambda_j\neq \Lambda$
is an invariant set, $j=1,3$. Let $E_2<0<E_4$ be close enough to $0\in \real$. Then the number of eigenvalues of $P_{\eps}$ in the rectangle
$[E_2,E_4] + i\eps [F_3,F_1]$ is equal to
$$
\frac{1}{(2\pi h)^2}{\rm vol}\left(\bigcup_{E_2\leq E\leq E_4} \Omega(E)\right)\left(1 + o(1)\right),\quad h\rightarrow 0.
$$
Here the set
$$
\Omega(0)= \{ \rho \in p^{-1}(0), Q_{\infty}(\Lambda(\rho))\subset [F_3,F_1]\}
$$
is flow-invariant, with $\partial \Omega(0) = \Lambda_1 \cup \Lambda_3$. The flow-invariant sets $\Omega(E)\subset p^{-1}(E)$, $E_2\leq E \leq E_4$ are
defined similarly --- see the precise statement of Theorem \ref{theo_main} below. We may say therefore that the imaginary parts of the eigenvalues of
$P_{\eps}$ are distributed according to a Weyl law, expressed in terms of the long-time averages of the leading non-selfadjoint
perturbation $q$ along the classical flow of $p$.

\medskip
\noindent
We would like to conclude the introduction by mentioning the work of Shni\-relman~\cite{Shnirelman} in the two-dimensional KAM-type situation, which
contains the idea of exploiting invariant Lagrangian tori to separate the real energy surface into different invariant regions, in order to study the
asymptotic multiplicity of the spectrum of the Laplacian. See also~\cite{Colin}. In our non-selfadjoint case, the idea of using
invariant tori as barriers for the flow becomes more efficient than in the standard selfadjoint setting, and the main result of this work can be considered as a
justification of this statement.

\medskip
\noindent
{\bf Acknowledgment}. We would like to thank San V\~u Ng\d{o}c for some helpful advice. Part of this projet was conducted when
the first author visited Universit\'e de Bourgogne in June and December of 2009. It is a
great pleasure for him to thank the Institut de Math\'ematiques de Bourgogne for a generous hospitality. The partial support of his
research by the National Science Foundation under grant DMS-0653275 and by the Alfred P. Sloan Research Fellowship is also gratefully acknowledged.
The second author acknowledges the support from a contract FABER of the Conseil r\'egional de Bourgogne as well as the contracts ANR-08-BLAN-0228-01
and JC05-52556 of the Agence Nationale de la Recherche.

\section{Assumptions and statement of main result}
\label{statement}
\setcounter{equation}{0}
\subsection{General assumptions}
We shall describe the general assumptions on our operators, which will be the same as
in~\cite{HiSjVu},~\cite{HiSj08}, as well as in the earlier papers in this series. Let $M$ stand for
either the space $\real^2$ or a real analytic compact manifold of dimension two. We shall let $\widetilde{M}$ denote a
complexification of $M$, so that $\widetilde{M}=\comp^2$ in the Euclidean case, and in the compact case, we let
$\widetilde{M}$ be a Grauert tube of $M$ --- see~\cite{GuSt} for the definition and further references.

When $M=\real^2$, let
\begeq
\label{eq1}
P_{\eps}=P^w(x,hD_x,\eps;h),\quad 0<h \leq 1,
\endeq
be the $h$--Weyl quantization on $\real^2$ of a symbol $P(x,\xi,\eps;h)$ (i.e. the Weyl quantization of
$P(x,h \xi,\eps;h)$), depending smoothly on $\eps\in {\rm neigh}(0,[0,\infty))$ and taking values in
the space of holomorphic functions of $(x,\xi)$ in a tubular \neigh{}
of $\real^4$ in $\comp^4$, with
\begeq
\label{eq2}
\abs{P(x,\xi,\eps;h)}\leq {\cal O}(1) m(\Re (x,\xi)),
\endeq
there. Here $m\geq 1$ is an order function on $\real^4$, in the sense that
\begeq
\label{eq3}
m(X)\leq C_0 \langle{X-Y}\rangle^{N_0} m(Y),\quad X,\,Y\in \real^4,
\endeq
for some fixed $C_0$, $N_0>0$. We shall assume, as we may, that $m$ belongs to its own symbol class, so that
$m\in C^{\infty}(\real^4)$ and
$\partial^{\alpha}m={\cal O}_{\alpha}(m)$ for each $\alpha\in \nat^4$.

\medskip
Assume furthermore that as $h\rightarrow 0$,
\begeq
\label{eq4}
P(x,\xi,\eps;h)\sim \sum_{j=0}^{\infty} h^j p_{j,\eps}(x,\xi)
\endeq
in the space of holomorphic functions satisfying (\ref{eq2}) in a
fixed tubular \neigh{} of $\real^4$. We make the basic assumption of ellipticity near infinity,
\begeq
\label{eq5}
\abs{p_{0,\eps}(x,\xi)}\geq \frac{1}{C} m(\Re (x,\xi)),\quad \abs{(x,\xi)}\geq C,
\endeq
for some $C>0$.

\medskip
When $M$ is a compact manifold, we shall take $P_{\eps}$ to be an $h$--differential operator on $M$,
such that for every choice of local coordinates, centered at some point of $M$, it takes the form
\begeq
\label{eq6}
P_{\eps}=\sum_{\abs{\alpha}\leq m} a_{\alpha,\eps}(x;h)(hD_x)^{\alpha},
\endeq
where $a_{\alpha,\eps}(x;h)$ is a smooth function of $\eps\in {\rm neigh}(0,[0,\infty))$ with values
in the space of bounded holomorphic functions in a complex \neigh{} of $x=0$. We further assume that
\begeq
\label{eq7}
a_{\alpha,\eps}(x;h)\sim \sum_{j=0}^{\infty} a_{\alpha,\eps,j}(x) h^j,\quad h\rightarrow 0,
\endeq
in the space of such functions. The semiclassical principal symbol $p_{0,\eps}$, defined on $T^*M$, takes the form
\begeq
\label{eq8}
p_{0,\eps}(x,\xi)=\sum a_{\alpha,\eps,0}(x)\xi^{\alpha},
\endeq
if $(x,\xi)$ are canonical coordinates on $T^*M$. We make the ellipticity assumption,
\begeq
\label{eq9}
\abs{p_{0,\eps}(x,\xi)}\geq \frac{1}{C} \langle{\xi}\rangle^m,\quad (x,\xi)\in T^*M,\quad \abs{\xi}\geq C,
\endeq
for some large $C>0$. Here we assume that $M$ has been equipped with some real analytic Riemannian metric, so
that $\abs{\xi}$ and $\langle{\xi}\rangle=(1+\abs{\xi}^2)^{1/2}$ are well-defined. We shall consider the operator $P_{\eps}$ as an unbounded
operator: $L^2(M,\mu(dx))\rightarrow L^2(M,\mu(dx))$, with the domain $H^m(M)$, the standard Sobolev space of order $m$. Here $\mu(dx)$ is the
Riemannian volume element on $M$.

\bigskip
Back in the Euclidean case, the ellipticity assumption (\ref{eq5}) implies that, for $h>0$ small enough and when
equipped with the domain
\begeq
\label{eq9.0}
H(m):=\left(m^w(x,hD)\right)^{-1}\left(L^2(\real^2)\right),
\endeq
the operator $P_{\eps}$ becomes closed and densely defined on $L^2(\real^2)$. We shall furthermore make the following assumption concerning the growth
of the order function $m(X)$ at infinity: for some $k\in \nat$, we have
\begeq
\label{eq9.1}
\frac{1}{m^k} \in L^1(\real^4).
\endeq
Since $P_{\eps}$ can be replaced by its $k$-th power $P_{\eps}^k$ in what follows, we may and will assume henceforth that $k=1$. The assumption that
\begeq
\label{eq9.2}
\frac{1}{m} \in L^1(\real^4)
\endeq
will prove very useful in various trace class considerations throughout the paper. In the compact case, it amounts to assuming that the order $m$
of $P_{\eps}$ satisfies $m>2$.

\medskip
\noindent
We shall write $p_{\eps}$ for $p_{0,\eps}$ and simply $p$ for $p_{0,0}$. We make the assumption that
\begeq
\label{eq9.3}
P_{\eps=0}\quad \wrtext{is formally selfadjoint}.
\endeq

\medskip
It follows from the assumptions (\ref{eq5}), (\ref{eq9}), and (\ref{eq9.3}) that the spectrum of $P_{\eps}$ in a fixed \neigh{} of $0\in \comp$ is
discrete, when $0<h\leq h_0$, $0\leq \eps\leq \eps_0$, with $h_0>0$, $\eps_0>0$ sufficiently small. Moreover, if
$z\in {\rm neigh}(0,\comp)$ is an eigenvalue of $P_{\eps}$ then $\Im z ={\cal O}(\eps)$.

\medskip
We furthermore assume that the real energy surface $p^{-1}(0)\cap T^*M$ is connected and that
\begeq
\label{eq9.4}
dp\neq 0\quad \wrtext{along}\quad p^{-1}(0)\cap T^*M.
\endeq

\bigskip
\noindent
We then have the following general result, stating that the distribution of the real parts the eigenvalues of $P_{\eps}$ near $0\in \comp$ is
governed by the usual Weyl law of the form (\ref{Weyl0}), with the following remainder estimate,
\begeq
\label{eq10}
\# \{z \in {\rm Spec}(P_{\eps}); \, E_1 \leq {\rm Re}\,z \leq E_2\} = \frac{1}{(2\pi h)^n}
\int\!\!\!\int_{p^{-1}([E_1,E_2])} dx\,d\xi + {\cal O}\left({\rm max}(\eps,h)h^{-n}\right).
\endeq
Here $E_j\in {\rm neigh}(0,\real)$, $j=1,2$, are close enough to $0$ but independent of $h$. In fact, in
(\ref{eq10}) we may allow $E_j$ to depend on $h$ with $E_2 - E_1 \geq c {\rm max}(\eps,h)$, for some $c>0$ fixed. The Weyl law (\ref{eq10}) is obtained
by following the general arguments of section 5 of~\cite{Sj00}, which are based on an adaptation to the semiclassical case of the ideas
in~\cite{Ma},~\cite{MaMa}. See also~\cite{A}.

Let $H_p=p'_{\xi}\cdot \partial_x-p'_x\cdot \partial_{\xi}$ be the Hamilton field of $p$. Throughout this paper, we shall work under the
assumption that the $H_p$--flow is completely integrable. Following~\cite{HiSj08}, let us now proceed to discuss the precise assumptions on the
geometry of the energy surface $p^{-1}(0)\cap T^*M$ in this case.

\subsection{Assumptions related to the complete integrability}

We assume that there exists an analytic real valued function $f$ defined near $p^{-1}(0)\cap T^*M$ such that $H_pf=0$,
with the differentials $df$ and $dp$ being linearly independent on an open and dense set $\subset {\rm neigh}(p^{-1}(0)\cap T^*M, T^*M)$.
For each $E\in {\rm neigh}(0,\real)$,  the level sets
$$
\Lambda_{\mu,E}=f^{-1}(\mu)\cap p^{-1}(E)\cap T^*M
$$
are invariant under the $H_p$--flow and form a singular foliation of the three-dimensional hypersurface $p^{-1}(E)\cap T^*M$. When
$(\mu,E)\in \real^2$ is such that $df\wedge dp\neq 0$ along $\Lambda_{\mu,E}$, then $\Lambda_{\mu,E}$ is a two-dimensional real analytic Lagrangian
submanifold of $T^*M$, which is a finite union of tori. In what follows
we shall use the word "leaf" and notation $\Lambda$ for a connected component of some $\Lambda_{\mu,E}$. Let $J$ be the set of all leaves in
$p^{-1}(0)\cap T^*M$. We have a disjoint union decomposition
\begeq
\label{eq12}
p^{-1}(0)\cap T^*M=\bigcup_{\Lambda \in J} \Lambda,
\endeq
where $\Lambda$ are compact flow--invariant connected sets.

The set $J$ has a natural structure of a graph whose edges correspond to families of regular leaves and the set $S$ of vertices is
composed of singular leaves. The union of edges $J\backslash S$ possesses a natural real
analytic structure and the corresponding Lagrangian tori depend analytically on $\Lambda\in J\backslash S$ with respect to that structure.

For simplicity, we shall assume that $f$ is a Morse-Bott function restricted to $p^{-1}(0)\cap T^*M$. In this case, the structure of
the singular leaves is known~\cite{BF}, \cite{San}, and the set $J$ is a finite connected graph. We shall identify each edge of $J$ analytically with a
real bounded interval and this determines a distance on $J$ in the natural way. The following continuity property holds,
\begin{eqnarray}
\label{13}
& & \wrtext{For every}\,\,\,\Lambda_0\in J\,\,\wrtext{and
every}\,\,\eps>0,\, \wrtext{there exists}\, \delta>0,\,\,\wrtext{such that if} \\ \nonumber
& & \Lambda\in J,\,\,{\rm
dist}_J(\Lambda,\Lambda_0)<\delta,\,\,\wrtext{then}\,\, \Lambda\subset
\{\rho\in p^{-1}(0)\cap T^*M;\, {\rm dist}(\rho,\Lambda_0)<\eps\}.
\end{eqnarray}

\medskip
By the Arnold-Mineur-Liouville theorem~\cite{D}, each torus $\Lambda\in J\backslash S$ carries real analytic
coordinates $x_1$, $x_2$ identifying $\Lambda$ with ${\bf T}^2 :=
\real^2/2\pi \z^2$, so that along $\Lambda$, we have
\begeq
\label{eq14}
H_p=a_1\partial_{x_1}+a_2\partial_{x_2},
\endeq
where $a_1$, $a_2\in \real$. The rotation number is defined as the ratio
\begeq
\label{eq14.1}
\omega(\Lambda)=[a_1:a_2]\in \real {\bf P}^1,
\endeq
and it depends analytically on $\Lambda\in J\backslash S$.
%We assume that
%$$
%\omega(\Lambda)\,\,\wrtext{is not identically constant on any open edge}.
%$$

\bigskip
\noindent
In what follows we shall write
\begeq
\label{eq11}
p_{\eps}=p+i\eps q+{\cal O}(\eps^2),
\endeq
in a \neigh{} of $p^{-1}(0)\cap T^*M$, and for simplicity we shall assume throughout this paper that
$q$ is real valued on the real domain. (In the general case, we should simply replace $q$ below by $\Re q$.)
For each torus $\Lambda\in J\backslash S$, we define the torus
average $\langle{q}\rangle(\Lambda)$ obtained by integrating $q|_{\Lambda}$ with respect to the natural smooth
measure on $\Lambda$, and assume that the analytic function $J\backslash S \ni \Lambda\mapsto \langle{q}\rangle(\Lambda)$
is not identically constant on any open edge. Here the integration measure used in the definition of $\langle{q}\rangle(\Lambda)$ comes
from the diffeomorphism between $\Lambda$ and ${\bf T}^2$, given by the action-angle coordinates.

\medskip
We introduce
\begeq
\label{eq15}
\langle{q}\rangle_T=\frac{1}{T} \int_{0}^{T} q\circ
\exp(tH_p)\,dt,\quad T>0,
\endeq
and consider the compact intervals $Q_{\infty}(\Lambda)\subset \real$, $\Lambda\in J$, defined as in~\cite{HiSjVu},~\cite{HiSj08},
\begeq
\label{eq16}
Q_{\infty}(\Lambda)=\left[\lim_{T \rightarrow \infty} \inf_{\Lambda}
\langle{q}\rangle_T, \lim_{T \rightarrow \infty}
\sup_{\Lambda}\langle{q}\rangle_T\right].
\endeq
Notice that when $\Lambda\in J\backslash S$ and $\omega(\Lambda)\notin {\bf Q}$ then
$Q_{\infty}(\Lambda)=\{\langle{q}\rangle(\Lambda)\}$. In the rational case,
we write $\omega(\Lambda)=\frac{m}{n}$, where $m\in \z$ and $n\in {\bf N}$ are
relatively prime, and where we may assume that $m={\cal
O}(n)$. When $k(\omega(\Lambda)):=\abs{m}+\abs{n}$ is the height of $\omega(\Lambda)$,
we recall from Proposition 7.1 in~\cite{HiSjVu} that
\begeq
\label{eq17}
Q_{\infty}(\Lambda)\subset \langle{q}\rangle(\Lambda)+{\cal
O}\left(\frac{1}{k(\omega(\Lambda))^{\infty}}\right)[-1,1].
\endeq

\medskip
\noindent
{\it Remark}. As $J\backslash S\ni \Lambda \rightarrow \Lambda_0\in S$,
the set of all accumulation points of $\langle{q}\rangle(\Lambda)$ is contained in the interval
$Q_{\infty}(\Lambda_0)$. Indeed, when $\Lambda\in J\backslash S$ and $T>0$, there exists
$\rho=\rho_{T,\Lambda}\in \Lambda$ such that
$\langle{q}\rangle(\Lambda)=\langle{q}\rangle_T(\rho)$. Therefore,
each accumulation point of $\langle{q}\rangle(\Lambda)$ as
$\Lambda\rightarrow \Lambda_0\in S$, belongs to
$[\inf_{\Lambda_0}\langle{q}\rangle_T,
\sup_{\Lambda_0}\langle{q}\rangle_T]$. The conclusion follows if we let
$T \rightarrow \infty$.

\subsection{Statement of the main result}

From Theorem 7.6 in~\cite{HiSjVu} we recall that
\begeq
\label{eq18}
\frac{1}{\eps}\Im \left({\rm Spec}(P_{\eps})\cap \{z; \abs{\Re z}\leq \delta\}\right)
\subset \left[\inf \bigcup_{\Lambda\in J} Q_{\infty}(\Lambda)-o(1),
\sup \bigcup_{\Lambda\in J} Q_{\infty}(\Lambda)+o(1)\right],
\endeq
as $\eps$, $h$, $\delta\rightarrow 0$. Let us also recall
from~\cite{HiSjVu} that a torus $\Lambda\in J\backslash S$ is said to be
Diophantine if representing $H_p|_{\Lambda} =
a_1\partial_{x_1}+a_2\partial_{x_2}$, as in (\ref{eq14}), we have
\begeq
\label{eq19}
\abs{a\cdot k}\geq \frac{1}{C_0 \abs{k}^{N_0}},\quad 0\neq k\in \z^2,
\endeq
for some fixed $C_0$, $N_0>0$.

Let
$$
F_j\in \bigcup_{\Lambda\in J} Q_{\infty}(\Lambda),\quad j=1,3,\quad F_3 < F_1,
$$
be such that there exist finitely many Lagrangian tori
\begeq
\label{eq20}
\Lambda_{1,j},\ldots\, ,\Lambda_{L_j,j}\in J\backslash S,\quad j=1,3,
\endeq
that are uniformly Diophantine as in (\ref{eq19}), and such that
\begeq
\label{eq21}
\langle{q}\rangle(\Lambda_{k,j})=F_j\quad \wrtext{for}\,\, 1\leq k\leq L_j,\quad j=1,3,
\endeq
with
\begeq
\label{eq22}
d_{\Lambda}\langle{q}\rangle(\Lambda_{k,j})\neq 0,\quad 1\leq k\leq L_j,\quad j=1,3.
\endeq

\bigskip
\noindent
We shall make the following global assumptions, for $j=1,3$:
\begeq
\label{eq23}
F_j \notin \bigcup_{\Lambda \in J \backslash \{\Lambda_{1,j},\ldots, \Lambda_{L_j,j}\}}Q_{\infty}(\Lambda).
\endeq
%From Lemma 2.4 in~\cite{HiSjVu} we recall that the set in the right hand side in (\ref{eq0.44}) is compact.
Here we recall from~\cite{HiSj08} that the earlier assumptions imply that for $j=1,3$,
$F_j \notin Q_{\infty}(\Lambda)$ for $\Lambda_{k,j}\neq \Lambda \in {\rm
neigh}(\Lambda_{k,j},J)$, $1\leq k\leq L_j$.

\medskip
\noindent
For notational simplicity only, throughout the following discussion we shall assume
that $L_1=L_3=1$ and we shall then write $\Lambda_1:=\Lambda_{1,1}$ and $\Lambda_3:=\Lambda_{1,3}$ for the corresponding
Diophantine tori $\subset p^{-1}(0)\cap T^*M$.

%belong to the same open edge of $J$, so that, when identifying the edge with a real bounded interval, we have
%\begeq
%\label{eq24}
%\Lambda_{3} < \Lambda_{1}.
%\endeq

\bigskip
\noindent
Assume that the strength of the non-selfadjoint perturbation $\eps$ satisfies
\begeq
\label{eq25.01}
h^K \leq \eps \leq h^{\delta},\quad \delta>0,
\endeq
for some fixed $K\geq 1$. In the work~\cite{HiSjVu}, under the assumptions (\ref{eq23}), (\ref{eq25.01}), for $h>0$ small enough,
we have determined all eigenvalues of $P_{\eps}$ mod ${\cal O}(h^{\infty})$, in a spectral window of the form
\begeq
\label{eq25}
\left[ -\frac{\eps^{\widetilde{\delta}}}{C}, \frac{\eps^{\widetilde{\delta}}}{C}\right] + i\eps \left[F_j - \frac{\eps^{\widetilde{\delta}}}{C},
F_j + \frac{\eps^{\widetilde{\delta}}}{C}\right], \quad j=1,3.
\endeq
Here $\widetilde{\delta}>0$ is sufficiently small and the constant $C>0$ is sufficiently large. Recall also that the eigenvalues of $P_{\eps}$
in the region (\ref{eq25}) form a distorted lattice, and their total number is
$$
\sim \frac{\eps^{2\widetilde{\delta}}}{h^2}.
$$

\medskip
\noindent
Our purpose here is to study the distribution of the {imaginary} parts of the eigenvalues of $P_{\eps}$ in suitable "large" sub-bands of the
entire spectral band
$$
\{z\in {\rm neigh}(0,\comp);\, \Im z = {\cal O}(\eps)\}.
$$
Specifically, we shall be interested in counting the number of eigenvalues of $P_{\eps}$ in a region of the form
\begeq
\label{eq25.1}
\left[-\frac{\eps^{\widetilde{\delta}}}{C}, \frac{\eps^{\widetilde{\delta}}}{C} \right] + i\eps \left[F_3,F_1\right],
\endeq
bounded from above and from below by the Diophantine levels $F_j$, $j=1,3$, introduced above.

\bigskip
\noindent
The following is the main result of this work.

\begin{theo}
\label{theo_main}
Let $P_{\eps}$ satisfy the assumptions made in subsections {\rm 2.1} and {\rm 2.2}, and in particular, {\rm (\ref{eq9.2})}.
Let $F_j\in \cup_{\Lambda\in J} Q_{\infty}(\Lambda)$, $j=1,3$, $F_3<F_1$, be such that the assumptions {\rm (\ref{eq21})}, {\rm (\ref{eq22})},
{\rm (\ref{eq23})} are satisfied. Let $0<\delta < 1 < K<\infty$ and assume that $h^K \leq \eps \leq h^{\delta}$. Let $E_2<0<E_4$
satisfy $\abs{E_j}\sim \eps^{\widetilde{\delta}}$, where $0<\widetilde{\delta}<1/3$ is small enough, so that
$$
h\leq \eps^{10\widetilde{\delta}}\log\frac{1}{\eps}.
$$
Then the number of eigenvalues of $P_{\eps}$ in the rectangle
\begeq
\label{eq25.2}
R  = \left(E_2, E_4\right) + i \eps (F_3,F_1),
\endeq
counted with their algebraic multiplicities, is equal to
\begeq
\label{Weyl}
\frac{1}{(2\pi h)^2}\left({\rm vol}\left(\bigcup_{E_2\leq E\leq E_4} \Omega(E)\right) +
{\cal O}\left(\eps^{3 \widetilde{\delta}}\log \frac{1}{\eps}\right)\right).
\endeq
Here the set
\begeq
\label{eq26}
\Omega(0)=\left\{\rho \in p^{-1}(0)\cap T^*M; Q_{\infty}(\Lambda(\rho))\subset [F_3,F_1]\right\},
\endeq
is flow-invariant, whose boundary is $\Lambda_1\cup \Lambda_3$. We define the set $\Omega(E)\subset p^{-1}(E)\cap T^*M$ to be the set which is close to $\Omega(0)$ with
boundary $\Lambda_1(E)\cup \Lambda_3(E)$ where $\Lambda_j(E)$ is the unique invariant torus $\subset p^{-1}(E)\cap T^*M$ close to $\Lambda_j$,
determined, thanks to {\rm (\ref{eq22})}, by the condition $\langle{q}\rangle(\Lambda_j(E))=F_j$, $j=1,3$.
\end{theo}

\medskip
\noindent
{\it Remark}. Combining Theorem \ref{theo_main} together with the Weyl law (\ref{eq10}), we see that as $h\rightarrow 0$, the ratio
$$
\frac{\#\left({\rm Spec}(P_{\eps})\cap R\right)}{\#\left({\rm Spec}(P_{\eps})\cap \left([E_2,E_4]+ i \real\right)\right)}
$$
converges to the expression
$$
\frac{{\int_{p^{-1}(0)} 1_{\Omega(0)}\, {\cal L}(d(x,\xi))}}{\int_{p^{-1}(0)} {\cal L}(d(x,\xi))}.
$$
Here ${\cal L}(d(x,\xi))$ stands for the Liouville measure on $p^{-1}(0)$.

\bigskip
\noindent
The result of Theorem \ref{theo_main} admits a natural extension to the case when the real energy $E\in [E_2,E_4]$ varies in a sufficiently small
but $h$--independent \neigh{} of $0\in \real$, which we now proceed to describe. When doing so, we shall assume that the tori $\Lambda_j$, $j=1,3$,
introduced in (\ref{eq20}), satisfy the following additional isoenergetic assumption,
\begeq
\label{eq26.01}
d_{\Lambda=\Lambda_j} \omega(\Lambda)\neq 0,\quad j=1,3.
\endeq
Here $\omega(\Lambda)$ is the rotation number of $\Lambda$, introduced in (\ref{eq14.1}). Let us represent $\Lambda_j \simeq \{\xi=0\}\subset T^* {\bf T}^2_x$ using the action-angle coordinates near $\Lambda_j$, so that $p=p(\xi)$,
$\omega(\xi) = [\partial_{\xi_1}p(\xi): \partial_{\xi_2}p(\xi)]$. It follows from (\ref{eq26.01}) that the map
$$
{\rm neigh}(0,\real^2)\ni \xi \mapsto \left(p(\xi),\omega(\xi)\right) \in \real \times \real {\bf P}^1
$$
is a local diffeomorphism. There exists therefore an analytic family of Lagrangian tori $\widetilde{\Lambda}_j(E)\subset p^{-1}(E)\cap T^*M$,
$E\in {\rm neigh}(0,\real)$, $j=1,3$, close to $\Lambda_j$, such that $\omega(\widetilde{\Lambda}_j(E)) = \omega(\Lambda_j)$, $j=1,3$. Let us set
when $E\in {\rm neigh}(0,\real)$,
\begeq
\label{eq26.02}
F_j(E) = \langle{q}\rangle(\widetilde{\Lambda}_j(E)),\quad j=1,3,
\endeq
and notice that an application of Theorem \ref{theo_main} together with the results of~\cite{HiSjVu} allow us to conclude that the number of eigenvalues
of $P_{\eps}$ in the region
\begeq
\label{eq26.03}
\left\{ \Re z\in [E_2,E_4];\, F_3(\Re z) \leq \frac{\Im z}{\eps} \leq F_1(\Re z)\right\},\quad E_2<0<E_4,\quad  \abs{E_j}\sim \eps^{\widetilde{\delta}},
\endeq
is given by
\begeq
\label{eq26.04}
\frac{1}{(2\pi h)^2}\left({\rm vol}\left(\bigcup_{E_2\leq E\leq E_4} \widetilde{\Omega}(E)\right) +
{\cal O}\left(\eps^{2\widetilde{\delta}}\right)\right).
\endeq
Here $\widetilde{\Omega}(0) = \Omega(0)$ in (\ref{eq26}) and the set $\widetilde{\Omega}(E)\subset p^{-1}(E)\cap T^*M$ is close to $\Omega(0)$,
with the boundary $\widetilde{\Lambda}_1(E)\cup \widetilde{\Lambda}_3(E)$. Covering a sufficiently small but fixed open interval $J \subset \real$
containing $0\in \real$ by ${\cal O}(\eps^{-\widetilde{\delta}})$ subintervals of length ${\cal O}(\eps^{\widetilde\delta})$ and applying the result of
Theorem \ref{theo_main} in the form (\ref{eq26.04}) uniformly as $E\in J$ varies, we get the following conclusion, by summing the individual
contributions from the subintervals.

\begin{theo}
\label{theo_ext}
Let us keep all the assumptions of Theorem {\rm \ref{theo_main}} and assume in addition that the isoenergetic condition {\rm (\ref{eq26.01})} holds. Let
$C>0$ be sufficiently large and let $E_2 < 0 < E_4$ be independent of $h$ with $\abs{E_j}\leq 1/C$, $j=2,4$. Introduce the functions
$$
F_j(E) = \langle{q}\rangle(\widetilde{\Lambda}_j(E)),\quad E\in {\rm neigh}(0,\real),
$$
for $j=1,3$, where $\widetilde{\Lambda}_j(E) \subset p^{-1}(E)\cap T^*M$ are Lagrangian tori close to $\Lambda_j$,
such that $\omega(\widetilde{\Lambda}_j(E)) = \omega(\Lambda_j)$, $j=1,3$. Then the number of eigenvalues of $P_{\eps}$ in the region
$$
\left \{ E_2 \leq \Re z \leq E_4,\,  F_3(\Re z) \leq \frac{\Im z}{\eps} \leq F_1(\Re z)\right\},
$$
counted with their algebraic multiplicities, is equal to
\begeq
\label{Weyl1}
\frac{1}{(2\pi h)^2}\left({\rm vol}\left(\bigcup_{E_2\leq E\leq E_4} \widetilde{\Omega}(E)\right) +
{\cal O}(\eps^{\widetilde{\delta}})\right).
\endeq
Here $0< \widetilde{\delta}< 1$ satisfies the same smallness condition as in Theorem {\rm \ref{theo_main}}.
\end{theo}

\bigskip
\noindent
In the course of the proof of Theorem \ref{theo_main}, we shall assume, as we may,
that the resolvent $(z-P_{\eps})^{-1}$ exists when $z\in \gamma := \partial R$. We equip $\gamma$ with the positive orientation, and
decompose
\begeq
\label{eq26.1}
\gamma=\gamma_1 \cup \gamma_2 \cup \gamma_3 \cup \gamma_4,
\endeq
where $\gamma_j$ for $j=1,3$, is the part of $\gamma$, where $\Im z = \eps F_j$. Correspondingly, we have
$\Re z = E_j$ along $\gamma_j$, when $j=2,4$. We are going to be concerned with the asymptotic behavior of the trace of
the spectral projection of $P_{\eps}$ associated to the spectrum of $P_{\eps}$ in $R$,
\begeq
\label{eq27}
{\rm tr}\, \frac{1}{2\pi i}\int_{\gamma} \left(z-P_{\eps}\right)^{-1}\,dz,
\endeq
giving the number of eigenvalues of $P_{\eps}$ in the rectangle $R$.

\bigskip
\noindent
The plan of the paper is as follows. In Section \ref{review}, we recall some basic facts and estimates, established in~\cite{HiSjVu}, to be used in the
proof of the main result. In particular, we recall the main features of the averaging procedure along the $H_p$--flow, effectively replacing $q$ in
(\ref{eq11}) by $\langle{q}\rangle_T$ in (\ref{eq15}), as well as the associated suitable IR-manifold $\subset T^*\widetilde{M}$, playing the role
of the real phase space throughout the proofs. Section \ref{vertical} is devoted to the analysis of the trace integrals along the vertical segments
$\gamma_2$, $\gamma_4$, contributing to the trace in (\ref{eq27}). Following the ideas and methods
of~\cite{Ma},~\cite{MaMa},~\cite{Sj97},~\cite{Sj00},~\cite{Sj01}, here
we make use of auxiliary trace class perturbations, constructed so that the perturbed non-selfadjoint operator has gaps in the spectrum. Let us also
notice that the analysis of Section \ref{vertical} is of a general nature and does not depend directly on the paper~\cite{HiSjVu}.
It is in Section \ref{horizontal}, concerned
with the resolvent integrals along the horizontal segments $\gamma_1$, $\gamma_3$, that the results of~\cite{HiSjVu} become important. The trace
analysis along the horizontal segments proceeds by means of a pseudodifferential partition of unity, and in particular, when understanding the
contributions coming from small neighborhoods of the Diophantine tori $\Lambda_j$, $j=1,3$, we apply the quantum Birkhoff normal form for $P_{\eps}$,
constructed in~\cite{HiSjVu}. Section \ref{horizontal} is concluded by combining the contributions from the different boundary segments
to derive a leading term for the trace integral (\ref{eq27}). The Weyl law (\ref{Weyl}) in Theorem \ref{theo_main} then follows when we let the
averaging time $T$ become sufficiently large. The brief Section \ref{periodic} is concerned with the task of deriving an analog of
Theorem \ref{theo_main} in the case when the Hamilton flow of $p$ is periodic on energy surfaces $p^{-1}(E)\cap T^*M$, $E\in {\rm neigh}(0,\real)$,
which was the main dynamical assumption in the series of works~\cite{HiSj04}--\cite{HiSj3a}.
A classical Hamiltonian with a periodic flow can be considered as a degenerate case of a completely integrable symbol, and thus, it seemed natural
to include this discussion here. In the final Section \ref{revolution}, we apply Theorems \ref{theo_main} and \ref{theo_ext} to a
complex perturbation of the semiclassical Laplacian on an analytic simple surface of revolution, and
complement the corresponding discussion in~\cite{HiSjVu}.

\section{Review of some results from~\cite{HiSjVu}}
\label{review}
\bigskip
\noindent
As alluded to in Section \ref{statement}, the analysis of the trace of the spectral projection in (\ref{eq27}) will proceed by working in the exponentially weighted
$h$-dependent Hilbert space $H(\Lambda)$, constructed and exploited in~\cite{HiSjVu}. The purpose of this section is therefore to recall briefly the
definition and the main features of the weighted space $H(\Lambda)$, following~\cite{HiSjVu}, as well as some resolvent estimates for $P_{\eps}$,
when viewed as an unbounded operator on $H(\Lambda)$.

%The following discussion will be
%concentrated on the case when $M=\real^2$, and we shall indicate at the end the modifications required when $M$ is compact.
\medskip
Following the discussion in Section 5 of~\cite{HiSjVu}, let us recall the following result.

\begin{prop}
\label{prop2.1}
Let us keep all the general assumptions of Section {\rm \ref{statement}}, and in particular {\rm (\ref{eq21})}, {\rm (\ref{eq22})}, and {\rm (\ref{eq23})}. Let
\begeq
\label{eq27.10}
\kappa_j : {\rm neigh}(\Lambda_j, T^*M)\rightarrow {\rm neigh}(\xi=0,T^*{\bf T}^2),\quad j=1,3,
\endeq
be the real and analytic canonical transformation given by the action-angle coordinates near $\Lambda_j$, such that $\kappa_j(\Lambda_j)=\{\xi=0\}$. Then
the leading symbol $p_{\eps}$ of $P_{\eps}$ expressed in terms of the coordinates $x$ and $\xi$ on the torus side, takes the form
$$
p_j(\xi) + i\eps q_j(x,\xi)+{\cal O}(\eps^2),\quad p_j(\xi) = a_j\cdot \xi +{\cal O}(\xi^2),\quad j=1,3.
$$
Define
$$
\langle{q_j}\rangle(\xi) = \frac{1}{(2\pi)^2} \int q_j (x,\xi)\,dx,\quad \xi\in {\rm neigh}(0,\real^2),
$$
so that $\langle{q_j}\rangle(0)=F_j$, and $dp_j(0)=a_j$ and $d\langle{q_j}\rangle(0)$ are linearly independent, for $j=1,3$.
Let $0<\widetilde{\eps}\ll 1$ be such that $\widetilde{\eps}\gg {\rm max}(\eps,h)$. Then there exists a globally defined smooth IR-manifold
$\Lambda \subset T^*\widetilde{M}$ and smooth Lagrangian tori $\widehat{\Lambda}_1$ and $\widehat{\Lambda}_3\subset \Lambda$,
such that when $\rho\in \Lambda$ is away from an $\widetilde{\eps}$-\neigh{} of $\widehat{\Lambda}_1\cup \widehat{\Lambda}_3$ in $\Lambda$ then
\begeq
\label{eq27.11}
\abs{\Re P_{\eps}(\rho)}\geq \frac{\widetilde{\eps}}{{\cal O}(1)}
\endeq
or
\begeq
\label{eq27.12}
\abs{\Im P_{\eps}(\rho)-\eps F_1}\geq \frac{\eps\widetilde{\eps}}{{\cal O}(1)}\quad \wrtext{and}\quad
\abs{\Im P_{\eps}(\rho)-\eps F_3}\geq \frac{\eps\widetilde{\eps}}{{\cal O}(1)}.
\endeq
The manifold $\Lambda$ is ${\cal O}(\eps)$-close to $T^*M$ and agrees with it away from a \neigh{} of $p^{-1}(0)\cap T^*M$. We have
$$
P_{\eps}  = {\cal O}(1): H(\Lambda,m)\rightarrow H(\Lambda).
$$
When $j=1,3$, there exists an elliptic $h$--Fourier integral operator
$$
U_j = {\cal O}(1): H(\Lambda)\rightarrow L^2_{\theta}({\bf T}^2),
$$
such that microlocally near $\widehat{\Lambda}_j$, we have
$$
U_j P_{\eps}  = \left(P_j^{(N)}(hD_x,\eps;h)+R_{N+1,j}(x,hD_x,\eps;h)\right)U_j.
$$
Here $P_j^{(N)}(hD_x,\eps;h)+R_{N+1,j}(x,hD_x,\eps;h)$ is defined microlocally near $\xi=0$ in $T^*{\bf T}^2$, the full symbol of
$P_j^{(N)}(hD_x,\eps;h)$ is independent of $x$, and
$$
R_{N+1,j}(x,\xi,\eps;h)  = {\cal O}((h,\xi,\eps)^{N+1}).
$$
Here the integer $N$ is arbitrarily large but fixed. The leading symbol of $P_j^{(N)}(hD_x,\eps;h)$ is of the form
$$
p_j(\xi) + i\eps \langle{q_j}\rangle(\xi) + {\cal O}(\eps^2).
$$
\end{prop}

\medskip
\noindent
{\it Remark}. In the work~\cite{HiSjVu}, the case of a single Diophantine level $F_j$ has been considered, with
the definition of the Hilbert space $H(\Lambda)$  depending on $F_j$. It is clear however, from the arguments of~\cite{HiSjVu}, that the construction
of the IR-manifold $\Lambda$ can be carried out so that it works for both $F_1$ and $F_3$, since the corresponding Diophantine tori
$\Lambda_j$, $j=1$, $3$, are disjoint. Let us also observe that the tori $\widehat{\Lambda}_j\subset \Lambda$ are ${\cal O}(\eps)$--close
to $\Lambda_j \subset p^{-1}(0)\cap T^*M$ in the $C^{\infty}$--sense, $j=1,3$.

\medskip
\noindent
{\it Remark}. Let us recall from~\cite{HiSjVu} that the space $L^2_{\theta}({\bf T}^2)$ of Floquet periodic functions, occurring in the statement of
Proposition 2.1,  consists of all $u\in L^2_{{\rm loc}}(\real^2)$ such that
$$
u(x-\nu) = e^{i\theta\cdot \nu} u(x),\quad \theta = \frac{S}{2\pi h} + \frac{k_0}{4},\quad \nu \in 2 \pi \z^2.
$$
Here $S=(S_1,S_2)$ is the pair of the classical actions computed along two suitable fundamental cycles in $\Lambda_j$, and the
2-tuple $k_0\in \z^2$ stands for the Maslov indices of the corresponding cycles.

\bigskip
\noindent
Let $G_T$ be an analytic function defined in a \neigh{} of $p^{-1}(0)\cap T^*M$, such that
\begeq
\label{eq27.12.1}
H_p G_T  = q -\langle{q}\rangle_T,
\endeq
where $T>0$ is large enough but fixed. As in~\cite{HiSjVu}, we solve (\ref{eq27.12.1}) by setting
$$
G_T = \int T J_T(-t) q\circ \exp(tH_p)\,dt,\quad J_T(t) = \frac{1}{T} J\left(\frac{t}{T}\right),
$$
where the function $J$ is compactly supported, piecewise linear, with
$$
J'(t) = \delta(t) - 1_{[-1,0]}(t).
$$
It follows from the results of~\cite{HiSjVu} that there exists a $C^{\infty}$ canonical transformation
\begeq
\label{eq27.12.2}
\kappa: {\rm neigh}(p^{-1}(0),T^*M)\rightarrow {\rm neigh}(p^{-1}(0),\Lambda),
\endeq
such that, in the case when $M=\real^2$,
\begeq
\label{eq27.12.3}
\kappa(\rho) = \rho + i\eps H_G(\rho) + {\cal O}(\eps^2),
\endeq
in the $C^{\infty}$--sense. Here the function $G\in C^{\infty}_0(T^*M)$ is such that in a \neigh{} of $p^{-1}(0)\cap T^*M$, away from a small but
fixed \neigh{} of $\Lambda_1 \cup \Lambda_3$, we have $G=G_T$. Near $\Lambda_j$, $j=1,3$, the function $G$ is of the form
$\widetilde{G}\circ \kappa_j^{-1}$, where
$\kappa_j$ is the canonical transformation near $\Lambda_j$ given by the action-angle variables, defined in (\ref{eq27.10}),
and $\widetilde{G}$ is an analytic function defined in a fixed \neigh{} of $\xi=0$, such that
$$
H_p \widetilde{G} = q-r,\quad r(x,\xi) = \langle{q}\rangle(\xi)+{\cal O}(\xi^N),
$$
where $N\in \nat$ is arbitrarily large but fixed. We refer to Section 2 of~\cite{HiSjVu} for the details of the construction of the weight function
$G$, also in the compact case.

\medskip
\noindent
{\it Remark}. From~\cite{HiSjVu}, we know that in a complex \neigh{} of $p^{-1}(0)\cap T^*M$, away from a small
\neigh{} of $\Lambda_1\cup \Lambda_3$, we have
$$
\Lambda  = \exp(i\eps H_{G_T})\left(T^*M\right).
$$

\bigskip
We now come to recall the definition of the Hilbert space $H(\Lambda)$. When doing so, let us first concentrate on the case when $M=\real^2$.
We shall then take the standard FBI-Bargmann transform
\begeq
\label{eq27.13}
Tu(x) =  C h^{-3/2} \int e^{-\frac{1}{2h}(x-y)^2} u(y)\,dy,\quad C>0,
\endeq
and remark that according to~\cite{Sj95}, for a suitable choice of $C>0$ in (\ref{eq27.13}), $T$ maps $L^2(\real^2)$ unitarily onto
$$
H_{\Phi_0}(\comp^2):={\rm Hol}(\comp^2)\cap L^2(\comp^2; e^{-\frac{2\Phi_0}{h}}L(dx)).
$$
Here $\Phi_0(x) = \frac{1}{2} \left(\Im x\right)^2$ and $L(dx)$ is the Lebesgue measure in $\comp^2$.

\medskip
\noindent
When viewing $T$ in (\ref{eq27.13}) as a Fourier integral operator with a complex quadratic phase, we introduce the associated complex linear
canonical transformation
\begeq
\label{eq27.14}
\kappa_T: (y,\eta)\mapsto (x,\xi)= (y-i\eta,\eta),
\endeq
mapping the real phase space $\real^4$ onto the linear IR-manifold
$$
\Lambda_{\Phi_0} = \left\{\left(x,\frac{2}{i}\frac{\partial \Phi_0}{\partial x}\right);\, x\in \comp^2\right\}.
$$
It was shown in~\cite{HiSjVu} that the representation
\begeq
\label{eq27.15}
\kappa_T(\Lambda) = \left\{\left(x,\frac{2}{i}\frac{\partial \Phi_{\eps}}{\partial x}\right);\, x\in \comp^2\right\} =: \Lambda_{\Phi_{\eps}}
\endeq
holds. Here the function $\Phi_{\eps}\in C^{\infty}(\comp^2;\real)$ is uniformly strictly plurisubharmonic, with $\Phi_{\eps}-\Phi_0$ compactly
supported, and such that uniformly on $\comp^2$,
$$
\Phi_{\eps}(x)-\Phi_0(x)={\cal O}(\eps),\quad \nabla\left(\Phi_{\eps}(x)-\Phi_0(x)\right)={\cal O}(\eps).
$$

\medskip
\noindent
The Hilbert space $H(\Lambda)$ is then defined so that it agrees with $L^2(\real^2)$ as a linear space, and it is equipped with the norm
\begeq
\label{eq27.16}
\norm{u}=\norm{u}_{H(\Lambda)}:=\norm{Tu}_{L^2_{\Phi_{\eps}}},\quad L^2_{\Phi_{\eps}} = L^2\left(\comp^2; e^{-\frac{2\Phi_{\eps}}{h}}L(dx)\right).
\endeq
Furthermore, the natural Sobolev space $H(\Lambda,m)\subset H(\Lambda)$, associated to $\Lambda$ and
the order function $m$ is introduced so $H(\Lambda,m)$ agrees with the space $H(m)$ in (\ref{eq9.0}) as a space, and it is equipped with the norm
\begeq
\label{eq27.17}
\norm{u}^2_{H(\Lambda,m)} = \int \abs{Tu(x)}^2 \widetilde{m}^2(x) e^{-\frac{2\Phi_{\eps}(x)}{h}}\, L(dx),
\endeq
where $\widetilde{m}=m\circ \kappa_T^{-1}$ is viewed as a function on $\comp^2_x$ in the natural way.

\medskip
\noindent
{\it Remark}. We notice that in view of (\ref{eq9.2}), the inclusion map $H(\Lambda,m)\rightarrow H(\Lambda)$ is of trace class.

\bigskip
\noindent
When defining the action of the operator $P_{\eps}$ on $H(\Lambda)$, we follow~\cite{MeSj1},~\cite{Sj82}, and perform a contour deformation in the
integral representation of $P_{\eps}$ on the FBI--Bargmann transform side. The operator $P_{\eps}$ then receives a leading symbol given by
$p_{\eps}|_{\Lambda}$. More generally, when $a\in S(\Lambda_{\Phi_{\eps}},\widehat{m})$, with an order function $\widehat{m}$ on
$\Lambda_{\Phi_{\eps}}\simeq \comp^2_x$, we would like to define the Weyl quantization of $\widetilde{a}=a\circ \kappa_T\in C^{\infty}(\Lambda)$,
acting on $H(\Lambda)$. To this end, it will be convenient to introduce a globally unitary $h$-Fourier integral operator with a complex phase,
\begeq
\label{eqU}
U=U_{\eps}: L^2(\real^2)\rightarrow H(\Lambda),\quad U_{\eps=0}=1,
\endeq
depending smoothly on $\eps$, associated to the canonical transformation in (\ref{eq27.12.2}), in order to reduce the quantization procedure on
$\Lambda$ to that of Weyl on $T^*\real^2$. When defining the unitary operator $U$ globally, we follow the procedure described in detail
in~\cite{Sj95},~\cite{MeSj2}. We can then introduce
\begeq
{\rm Op}_h(\widetilde{a}) = {\cal O}(1): H(\Lambda,\widehat{m})\rightarrow H(\Lambda),
\endeq
defined as
\begeq
\label{eqKvant}
{\rm Op}_h(\widetilde{a}) = U\circ \left(\widetilde{a}\circ \kappa\right)^{w}(x,hD_x)\circ U^*,
\endeq
using the $h$--Weyl quantization on $\real^{2}$. In what follows, when quantizing symbols defined on the IR-manifold $\Lambda$,
we shall always use the unitary operator $U$ to reduce to the standard phase space $T^*\real^2$.
%either using the method of contour deformations (and suitable almost analytic extensions), or through the Toeplitz
%quantization approach, introducing and working with the orthogonal projection: $L^2_{\Phi_{\eps}}\rightarrow H_{\Phi_{\eps}}$.

\medskip
\noindent
{\it Remark}. In the case when $M$ is compact, the definition of the spaces $H(\Lambda)$ and $H(\Lambda,m) = H(\Lambda, \langle{\alpha_{\xi}}\rangle^m)$
has been given in the appendix of~\cite{HiSj04}, following~\cite{Sj96}. We refer to the latter paper for a discussion of the action of $h$-differential
operators with analytic coefficients on $H(\Lambda)$. When quantizing a symbol $a\in C^{\infty}_0(\Lambda)$, we follow~\cite{Sj96} and use the Toeplitz quantization, as explained in that paper.

\bigskip
\noindent
The following result is a consequence of Proposition \ref{prop2.1} and the arguments of Section 5 of~\cite{HiSjVu}. See also Proposition
6.3 of~\cite{HiSj08}.
\begin{prop}
\label{prop2.2}
Assume that {\rm (\ref{eq23})} holds and that $h^K\leq \eps \leq h^{\delta}$, $0<\delta <1 < K$. Let
$0<\widetilde{\delta}\ll 1$ be sufficiently small and let $N_0\geq 1$ be fixed. Assume that
$$
z\in \left[ -\frac{\eps^{\widetilde{\delta}}}{C}, \frac{\eps^{\widetilde{\delta}}}{C}\right] + i\eps \left[F_j - \frac{\eps^{\widetilde{\delta}}}{C},
F_j + \frac{\eps^{\widetilde{\delta}}}{C}\right], \quad j=1,3,
$$
is such that ${\rm dist}(z,{\rm Spec}(P_{\eps}))\geq \eps h^{N_0}$. We then have, in the sense of linear continuous operators: $H(\Lambda)\rightarrow H(\Lambda,m)$,
\begeq
\label{eq27.1}
\norm{\left(z-P_{\eps}\right)^{-1}} \leq \frac{{\cal O}(1)}{\eps h^{N_0}}.
\endeq
\end{prop}

\section{Trace class per\-tur\-bations and the vertical bo\-un\-da\-ry segments}
\label{vertical}
\setcounter{equation}{0}

In this section we shall be concerned with the trace of the integrals
$$
\frac{1}{2\pi i}\int_{\gamma_j} \left(z-P_{\eps}\right)^{-1}\,dz,\quad j=2,4,
$$
and, to fix the ideas, we shall take $j=2$. It will be clear that the treatment of the case $j=4$ is similar. Furthermore, for simplicity, we shall
concentrate on the case when $M=\real^2$.

\medskip
\noindent
In the Hilbert space $H(\Lambda)$, let us write
$$
P_{\eps}  = P + {\cal O}(\eps): H(\Lambda,m)\rightarrow H(\Lambda),
$$
where $P = U P_{\eps=0} U^*$ is selfadjoint in $H(\Lambda)$. Here $P_{\eps=0}$ is selfadjoint on $L^2(\real^2)$.
We shall write $p\in C^{\infty}(\Lambda)$ for the leading symbol of $P$.

\medskip
\noindent
Following the approach of~\cite{Ma},\cite{Sj00}, we shall introduce a trace class perturbation of $P_{\eps}$ in order to create a gap in the spectrum.
To this end, let $\chi\in C^{\infty}_0(\real;[0,1])$ be such that $\chi(t)=1$ for $\abs{t}\leq 1$, and consider the symbol
$$
k(\rho;h) = i \chi\left(\frac{p(\rho)-E_2}{\eps^{3\widetilde{\delta}}}\right),\quad \rho \in \Lambda.
$$
We shall assume that $0<\widetilde{\delta}< 1$ is so small that
\begeq
\label{eq27.2}
\eps^{3\widetilde{\delta}} \geq h^{\frac{1}{2}-\eta},
\endeq
for some $\eta>0$. Associated to $k(\rho;h)$, let us introduce
\begeq
\label{eq28}
\widetilde{P} = P + \eps^{3\widetilde{\delta}}K,
\endeq
where $K$ is the Weyl quantization of $k(\rho;h)$, obtained by using the $h$-Weyl quantization on $\real^2$, as described in (\ref{eqKvant}).

\medskip
The hypothesis (\ref{eq27.2}) together with the standard estimates for the operator norm and the trace class norm of a pseudodifferential operator on
$\real^n$, given in~\cite{DiSj}, show that
\begeq
\label{eq28.1}
\norm{\widetilde{P}-P}\leq {\cal O}\left(\eps^{3\widetilde{\delta}}\right),\quad \norm{\widetilde{P}-P}_{{\rm tr}}\leq
{\cal O}\left(\frac{\eps^{6\widetilde{\delta}}}{h^{2}}\right).
\endeq

\medskip
\noindent
Write next
$$
\widetilde{P}_{\eps} = P_{\eps} + \eps^{3\widetilde{\delta}} K.
$$
We shall make use of the following crude parametrix construction for $z-\widetilde{P}_{\eps}$, when $\Re z = E_2$ and $\Im z ={\cal O}(\eps)$.
In doing so let us consider
\begeq
\label{eq28.2}
e(\rho,z,\eps) = \frac{1}{z-\widetilde{p}_{\eps}(\rho)},\quad \widetilde{p}_{\eps}(\rho) := p_{\eps}(\rho) + i\eps^{3\widetilde{\delta}}
\chi\left(\frac{p(\rho)-E_2}{\eps^{3\widetilde{\delta}}}\right).
\endeq
Here $p_{\eps}$ is the leading symbol of $P_{\eps}$, acting on $H(\Lambda)$. We shall first restrict the attention to the region where
\begeq
\label{eq28.21}
\abs{p(\rho)-E_2} \leq {\cal O}(\eps^{3\widetilde{\delta}}).
\endeq
Here we have, with $\Re z =E_2$, $\Im z = {\cal O}(\eps)$,
\begeq
\label{eq28.22}
\abs{\widetilde{p}_{\eps}(\rho)-z}\geq \frac{\eps^{3\widetilde{\delta}}}{{\cal O}(1)}.
\endeq
Considering the usual expression for $\nabla^{\ell} e$, $\ell\geq 1$, given by the Fa\`a di Bruno's formula~\cite{Lerner}, we see that we have to estimate
the expression
\begeq
\label{eq28.23}
\frac{1}{(z-\widetilde{p}_{\eps})}\prod_{j=1}^k \frac{\nabla^{\ell_j} \widetilde{p}_{\eps}}{(z-\widetilde{p}_{\eps})},\quad \ell_j\geq 1, \quad \ell_1+\ldots\, +\ell_k=\ell.
\endeq
Distinguishing the cases $\ell_j=1$ and $\ell_j \geq 2$, and using that $\nabla \widetilde{p}_{\eps}={\cal O}(1)$ together with (\ref{eq28.22}),
we obtain that in the region where (\ref{eq28.21}) holds, we have
\begeq
\label{eq28.3}
\nabla^{\ell} e(\rho,z,\eps) = \eps^{-3\widetilde{\delta}} {\cal O}_{\ell}\left(\eps^{-3\widetilde{\delta}\ell}\right).
\endeq

\medskip
\noindent
When $\rho \in \Lambda$ is in a bounded set away from the region where (\ref{eq28.21}) is valid, the estimate (\ref{eq28.3}) improves to the following,
\begeq
\label{eq28.4}
\nabla^{\ell} e(\rho,\eps,z) = {\cal O}\left(\frac{1}{\abs{p_{\eps}(\rho)-z}^{1+\ell}}\right).
\endeq
Finally, when $\rho$ is in a \neigh\, of infinity, where $\Lambda$ agrees with $\real^4$, we get
\begeq
\label{eq28.5}
\nabla^{\ell} e(\rho,\eps,z) = {\cal O}_{\ell}\left(\frac{1}{m(\rho)}\right).
\endeq
Introducing $E(z)$ to be the Weyl quantization of $e(\rho,z,\eps)$, defined using the unitary map (\ref{eqU}), as described in Section \ref{review},
we conclude, using (\ref{eq27.2}), (\ref{eq28.3}), (\ref{eq28.4}), (\ref{eq28.5}), that
\begeq
\label{eq28.6}
E(z) = {\cal O}\left(\frac{1}{\eps^{3\widetilde{\delta}}}\right): H(\Lambda)\rightarrow H(\Lambda,m),
\endeq
while the trace class norm of the operator $E(z)$ on $H(\Lambda)$ satisfies
\begeq
\label{eq28.7}
\norm{E(z)}_{{\rm tr}}\leq \frac{{\cal O}(1)}{h^2} \int\!\!\!\int_{K} \frac{\mu(d\rho)}{{\rm max}(\abs{p_{\eps}(\rho)-z},\eps^{3\widetilde{\delta}})}
\leq \frac{{\cal O}(1)}{h^2} \log\frac{1}{\eps}.
\endeq
Here $\mu(d\rho)$ is the symplectic volume element on the IR-manifold $\Lambda$, so that
$$
\mu(d\rho) = \frac{\sigma^2}{2!}\biggl|_{\Lambda},
$$
where $\sigma$ is the complex symplectic $(2,0)$-form on $T^*\widetilde{M}$, and $K\subset \Lambda$ is a sufficiently large fixed compact set.

\medskip
\noindent
It is then clear from (\ref{eq27.2}), (\ref{eq28.3}), (\ref{eq28.4}), and (\ref{eq28.5}), that
\begeq
\label{eq28.9}
(z-\widetilde{P}_{\eps})E(z) = 1 + R,
\endeq
where
\begeq
\label{eq28.10}
R = {\cal O}\left(\frac{h}{\eps^{6\widetilde{\delta}}}\right): H(\Lambda)\rightarrow H(\Lambda).
\endeq
It follows therefore that in the region where $\Re z = E_2$, $\Im z = {\cal O}(\eps)$, the operator
$z-\widetilde{P}_{\eps}: H(\Lambda,m)\rightarrow H(\Lambda)$ is bijective, and
\begeq
\label{eq28.101}
\left(z-\widetilde{P}_{\eps}\right)^{-1} = {\cal O}\left(\frac{1}{\eps^{3\widetilde{\delta}}}\right): H(\Lambda)\rightarrow H(\Lambda,m).
\endeq
Furthermore,
$$
\left(z-\widetilde{P}_{\eps}\right)^{-1} = E(z)(1+R)^{-1},
$$
and writing $(1+R)^{-1} = 1-(1+R)^{-1}R$, we get
\begeq
\label{eq28.11}
\left(z-\widetilde{P}_{\eps}\right)^{-1} - E(z) = -E(z)(1+R)^{-1}R.
\endeq
It follows therefore from (\ref{eq28.7}) and (\ref{eq28.10}) that the trace class norm of the operator in the left hand side of (\ref{eq28.11}) is
\begeq
\label{eq28.12}
{\cal O}\left(\frac{h \eps^{-6\widetilde{\delta}}}{h^2}\log\frac{1}{\eps}\right),
\endeq
assuming that (\ref{eq27.2}) holds.

\medskip
\noindent
We conclude that
\begeq
\label{eq28.13}
{\rm tr}\,\frac{1}{2\pi i}\int_{\gamma_2} \left(z-\widetilde{P}_{\eps}\right)^{-1}\,dz  =
{\rm tr}\,\frac{1}{2\pi i}\int_{\gamma_2} E(z)\,dz + {\cal O}\left(\frac{h \eps^{-5\widetilde{\delta}}}{h^2}\log\frac{1}{\eps}\right).
\endeq
Here we have also used that the length of $\gamma_2$ is ${\cal O}(\eps)$.

\bigskip
\noindent
We shall now compare traces of the integrals of the resolvents of $P_{\eps}$ and $\widetilde{P}_{\eps}$, following some classical methods in
non-selfadjoint spectral theory. Let us also remark that such methods now also have a tradition in the theory of resonances.
See~\cite{GoKr},~\cite{Ma},~\cite{MaMa},~\cite{Sj97},~\cite{Sj01},~\cite{PeZw},~\cite{Vodev},~\cite{Zw}.
We shall therefore only recall the main features of the argument, referring to the above mentioned works for the details.

\medskip
\noindent
Assume that $\Re z= E_2$ and $\Im z = {\cal O}(\eps)$. An application of the resolvent identity shows that
$$
(z-P_{\eps})^{-1}-(z-\widetilde{P}_{\eps})^{-1} = -\left(z-P_{\eps}\right)^{-1} \eps^{3\widetilde{\delta}} K \left(z-\widetilde{P}_{\eps}\right)^{-1}
$$
is of trace class on $H(\Lambda)$, and using the cyclicity of the trace, we obtain, by a classical calculation, that
$$
{\rm tr} \left(\left(z-P_{\eps}\right)^{-1} - \left(z-\widetilde{P}_{\eps}\right)^{-1}\right) =
{\rm tr} \left( \left(1+\widetilde{K}(z)\right)^{-1} \partial_z \widetilde{K}(z)\right).
$$
Here $\widetilde{K}(z) = \eps^{3\widetilde{\delta}}K\left(z-\widetilde{P}_{\eps}\right)^{-1}$. Hence,
\begin{multline*}
{\rm tr}\, \left(\frac{1}{2\pi i}\int_{\gamma_2}(z-P_{\eps})^{-1}\,dz  -
\frac{1}{2\pi i}\int_{\gamma_2}(z-\widetilde{P}_{\eps})^{-1}\,dz\right) \\ = \frac{1}{2\pi i}\int_{\gamma_2}
\partial_z \log {\rm det}\, \left(1+\eps^{3\widetilde{\delta}}K(z-\widetilde{P}_{\eps})^{-1}\right)\,dz.
\end{multline*}
We shall be interested in the real part of the expression above, which is equal to
$$
\frac{1}{2\pi} {\rm var\, arg}_{\gamma_2}\,D(z),
$$
where we have set
\begeq
\label{eq28.135}
D(z)={\rm det}\, \left(1+\widetilde{K}(z)\right).
\endeq
An application of (\ref{eq28.1}) and (\ref{eq28.101}) shows that
\begeq
\label{eq28.14}
\norm{\widetilde{K}(z)}_{{\rm tr}} \leq {\cal O}\left(\frac{\eps^{3\widetilde{\delta}}}{h^2}\right),
\endeq
in the region where $\Re z = E_2$, $\Im z = {\cal O}(\eps)$.

\medskip
\noindent
When estimating the argument variation of $D(z)$ along $\gamma_2$, we shall proceed by following the now well established and essentially classical
complex analytic argument, described in detail in~\cite{Ma},~\cite{Sj97},~\cite{Sj01}. (See also~\cite{Viola}.) In order to recall its main features,
let us consider the holomorphic determinant $D(z)$ in (\ref{eq28.135}) in a region of the form
\begeq
\label{eq29}
{R}_d = [E_2 - d \eps^{3\widetilde{\delta}},E_2 + d\eps^{3\widetilde{\delta}}]+i [-d\eps^{3\widetilde{\delta}},d \eps^{3\widetilde{\delta}}],
\quad d>0,
\endeq
and notice that the bound (\ref{eq28.14}) remains valid for $z\in R_d$. It follows that
\begeq
\label{eq29.1}
\abs{D(z)}\leq \exp(\norm{\widetilde{K}(z)}_{{\rm tr}}) \leq \exp\left({\cal O}(1)\frac{\eps^{3\widetilde{\delta}}}{h^2}\right),\quad z\in R_d.
\endeq
Let now $z_0 \in R_d$ be such that $\Im z_0<0$ and $\abs{\Im z_0}\geq d_1 \eps^{3\widetilde{\delta}}$, $d_1 < d$. We then have
$$
\left(z_0-P_{\eps}\right)^{-1}  = {\cal O}\left(\frac{1}{\eps^{3\widetilde{\delta}}}\right): H(\Lambda) \rightarrow H(\Lambda),
$$
and therefore, since
$$
D(z_0)^{-1} = {\rm det}\, \left((1+\widetilde{K}(z_0))^{-1}\right),
$$
with
$$
(1+\widetilde{K}(z_0))^{-1} = 1-\eps^{3\widetilde{\delta}} K (z_0-P_{\eps})^{-1},
$$
it follows that
\begeq
\label{eq29.2}
\abs{D(z_0)}\geq \exp\left(-{\cal O}(1) \frac{\eps^{3\widetilde{\delta}}}{h^2}\right).
\endeq

\noindent
Let $N=N(P_{\eps},R_d,h)$ be the number of eigenvalues $z_j$, $j=1,\ldots\, N$, of $P_{\eps}$ in $R_d$, repeated according to their multiplicity. Using
that (\ref{eq28.14}) continues to be valid in a region of the form $R_d$, with a slightly larger value of $d$, and combining this with (\ref{eq29.2})
and Jensen's formula, we obtain that
\begeq
\label{eq29.3}
N(P_{\eps},R_d,h) = {\cal O}(1) \frac{\eps^{3\widetilde{\delta}}}{h^2}.
\endeq
Proceeding further as in~\cite{Ma},~\cite{Sj97},~\cite{Sj01}, one next considers a factorization
\begeq
\label{eq29.4}
D(z) = G(z) \prod_{j=1}^N (z-z_j),\,\, z\in {R}_d,
\endeq
where $G$ and $1/G$ are holomorphic in $R_d$. An application of Cartan's lemma (or, alternatively, of Lemma 4.3 in~\cite{Sj01}) together with
the maximum principle and the Harnack inequality allows us to show that after an arbitrarily small decrease of $d>0$, we have
$$
\abs{\log\abs{G(z)}}\leq {\cal O}\left(\frac{\eps^{3\widetilde{\delta}}}{h^2}\right),\quad z\in R_d.
$$
It follows then easily (see, for instance, Lemma 1.8 in~\cite{Ma}) that the argument variation of $G(z)$ along $\gamma_2$ is
$$
{\rm var\, arg}_{\gamma_2}\,G(z) = {\cal O}\left(\frac{\eps^{3\widetilde{\delta}}}{h^2}\right).
$$
Combining this with (\ref{eq29.3}) and (\ref{eq29.4}), we obtain that
$$
{\rm var\, arg}_{\gamma_2}\,D(z) = {\cal O}\left(\frac{\eps^{3\widetilde{\delta}}}{h^2}\right).
$$

\begin{prop}
\label{prop4.1}
Assume that $\widetilde{\delta}>0$ is small enough so that $h^{\frac{1}{2}-\eta}\leq \eps^{3\widetilde{\delta}}$, for some $\eta>0$. We have
\begeq
\Re \left({\rm tr}\, \left(\frac{1}{2\pi i}\int_{\gamma_2}(z-P_{\eps})^{-1}\,dz  -
\frac{1}{2\pi i}\int_{\gamma_2}(z-\widetilde{P}_{\eps})^{-1}\,dz\right)\right) = {\cal O}\left(\frac{\eps^{3\widetilde{\delta}}}{h^2}\right).
\endeq
\end{prop}

\bigskip
\noindent
Combining (\ref{eq28.13}) and Proposition \ref{prop4.1}, we obtain that
\begeq
\label{eq30.101}
\Re {\rm tr}\, \frac{1}{2\pi i}\int_{\gamma_2} \left(z-P_{\eps}\right)^{-1}\,dz = \Re {\rm tr}\, \frac{1}{2\pi i}\int_{\gamma_2} E(z)\,dz +
{\cal O}\left(\frac{\eps^{3\widetilde{\delta}}}{h^2}+ \frac{h\eps^{-5\widetilde{\delta}}}{h^2}\log\frac{1}{\eps}\right),
\endeq
and here the remainder in the right hand side is
$$
{\cal O}\left(\frac{\eps^{3\widetilde{\delta}}}{h^2}\right),
$$
provided that $\widetilde{\delta}>0$ is so small that
\begeq
\label{eq30.1011}
h \leq \frac{\eps^{8\widetilde{\delta}}}{\abs{\log \eps}}.
\endeq
Notice that the smallness condition (\ref{eq27.2}) is implied by (\ref{eq30.1011}), provided that $\eta>0$ in (\ref{eq27.2}) is sufficiently small.

\bigskip
\noindent
Since the operator $E(z)$ is introduced by means of the Weyl quantization on $\real^2$ and the unitary map (\ref{eqU}), associated to the canonical
transformation (\ref{eq27.12.2}), we know that the trace of the trace class operator
$E(z): H(\Lambda)\rightarrow H(\Lambda)$ is given by
$$
{\rm tr}\, E(z) = \frac{1}{(2\pi h)^2}\int\!\!\!\int \frac{1}{z-\widetilde{p}_{\eps}(\rho)}\,\mu(d\rho).
$$
Here we recall that $\widetilde{p}_{\eps}$ has been introduced in (\ref{eq28.2}).

\medskip
\noindent
With (\ref{eq30.101}) in mind, we shall now compare the expressions
\begeq
\label{eq30.102}
\frac{1}{2\pi i}\int_{\gamma_2} \int\!\!\!\int \frac{1}{z-p_{\eps}(\rho)}\,\mu(d\rho)\,dz
\endeq
and
\begeq
\label{eq30.103}
\frac{1}{2\pi i} \int_{\gamma_2} \int\!\!\!\int \frac{1}{z-\widetilde{p}_{\eps}(\rho)}\,\mu(d\rho)\,dz.
\endeq
The difference between the integrals (\ref{eq30.102}) and (\ref{eq30.103}) is equal to
\begeq
\label{eq31}
\frac{-1}{2\pi i}\int_{\gamma_2} \int\!\!\!\int
\frac{i\eps^{3\widetilde{\delta}}\chi\left((p-E_2)/\eps^{3\widetilde{\delta}}\right)}{(z-\widetilde{p}_{\eps})(z-p_{\eps})}\,\mu(d\rho)\,dz,
\endeq
which, in view of (\ref{eq28.22}), does not exceed
$$
{\cal O}(1) \int_{\gamma_2} \int\!\!\!\int \frac{\chi((p-E_2)/\eps^{3\widetilde{\delta}})}{\abs{z-p_{\eps}}}\,\mu(d\rho)\,\abs{dz},
$$
which can in turn be estimated by
$$
{\cal O}(1) \int\!\!\!\int \chi\left(\frac{p-E_2}{\eps^{3\widetilde{\delta}}}\right) \left(-\log\abs{p-E_2}\right)\,\mu(d\rho)=
{\cal O}(1)\int_0^{\eps^{3\widetilde{\delta}}} -\log{t}\,dt = {\cal O}(1) \eps^{3\widetilde{\delta}} \log{\frac{1}{\eps}}.
$$

\bigskip
\noindent
We summarize the discussion in this section in the following proposition.

\begin{prop}
\label{prop4.2}
Let $E_2 < 0 < E_4$ be such that $\abs{E_j}\sim \eps^{\widetilde{\delta}}$, $j=2,4$, where $0<\widetilde{\delta}<1$ is so small that
$$
h \leq \frac{\eps^{8\widetilde{\delta}}}{\abs{\log \eps}}.
$$
When $\gamma_j$ is the vertical segment given by $\Re z = E_j$,
$\eps F_3 \leq \Im z \leq \eps F_1$, we have, for $j=2,4$,
\begin{multline*}
\Re {\rm tr}\, \frac{1}{2\pi i}\int_{\gamma_j} \left(z-P_{\eps}\right)^{-1}\,dz  \\ =
\Re \frac{1}{2\pi i} \frac{1}{(2\pi h)^2}\int_{\gamma_j} \int\!\!\!\int \frac{1}{z-p_{\eps}(\rho)}\, \mu(d\rho)\,dz +
\frac{1}{h^2} {\cal O}\left(\eps^{3\widetilde{\delta}} \log{\frac{1}{\eps}}\right).
\end{multline*}
\end{prop}

\medskip
\noindent
{\it Remark}. As mentioned above, the idea of using trace class perturbations to create a gap in the spectrum of a
non-selfadjoint operator has a long tradition in non-selfadjoint spectral theory,~\cite{MaMa},~\cite{Sj00}. Here we have chosen to create gaps
that are wide enough, so that simple pseudodifferential perturbations can be employed to that end. The price to pay for this simplicity is that the
remainder estimates that one obtains in the trace analysis in Proposition \ref{prop4.2} are not expected to be sharp, and it is quite
likely that finer estimates are possible to derive, at the expense of a greater technical investment. We hope to be able to return to
this question in a future paper.

%We shall also study the contribution of the term $r(\rho,z,\eps;h)$ in (\ref{eq30.11}), given by
%$$
%\frac{1}{2\pi i}\int_{\gamma_2} \int\!\!\!\int r(\rho,z,\eps;h)\,\mu(d\rho)\,dz.
%$$
%In view of (\ref{eq30.2}), it does not exceed
%$$
%{\cal O}(h) \int_{\gamma_2} \int\!\!\!\int \frac{\abs{dz} \mu(d\rho)}{\abs{z-\widetilde{p}_{\eps}}^2},
%$$
%which can in turn be estimated by
%$$
%{\cal O}(h) \int\!\!\!\int \frac{\mu(d\rho)}{\eps^{3\delta} + d(\widetilde{p}_{\eps}(\rho))} = {\cal O}(h) \log \frac{1}{\eps}.
%$$
%Here $d(\widetilde{p}_{\eps}(\rho))$ is the distance from $\widetilde{p}_{\eps}(\rho)$ to $\gamma_2$.

\section{Trace analysis near the Diophantine levels}
\label{horizontal}
\setcounter{equation}{0}
The purpose of this section is to understand the semiclassical behavior of the trace integrals
$$
{\rm tr}\,\frac{1}{2\pi i} \int_{\gamma_1} (z-P_{\eps})^{-1}\, dz,\quad j=1,3,
$$
and when doing so, we shall concentrate on the case when $j=1$. Here we recall that $\gamma_1$ is given by $\Im z=\eps F_1$,
$E_2\leq \Re z \leq E_4$, $\abs{E_j}\sim \eps^{\widetilde{\delta}}$, $j=2,4$. We shall assume throughout
that $0<\widetilde{\delta}<1$ satisfies
\begeq
\label{eq31.01}
h\leq {\cal O}(\eps^{9\widetilde{\delta}}),
\endeq
so that Proposition \ref{prop4.2} is applicable. Recall also that $F_3 < F_1$. We may, and will assume, in the remainder of the following discussion that
$z\in \gamma_1$ satisfies
\begeq
\label{eq31.1}
{\rm dist}(z,{\rm Spec}(P_{\eps}))\geq \eps h^{N_0},
\endeq
for some fixed $N_0\geq 1$, so that by Proposition \ref{prop2.2}, the resolvent $(z-P_{\eps})^{-1}$ satisfies an estimate of the form (\ref{eq27.1}).

\subsection{Trace integrals away from the Diophantine tori}

Following~\cite{HiSjVu}, thanks to Proposition \ref{prop2.1}, we shall consider a smooth partition of unity on the manifold $\Lambda$, given by
\begeq
\label{eq32}
1 = \chi_{1} + \chi_{3} + \psi_{{r},+} + \psi_{{r},-} + \psi_{{i},-} + \psi_{{i},0} + \psi_{{i},+}.
\endeq
Here $0 \leq \chi_{j}\in C^{\infty}_0(\Lambda)$ is a cut-off function to an $\eps^{{\widetilde{\delta}}}$-\neigh{} of
$\widehat{\Lambda}_j$, $j=1, 3$, such that, in the sense of trace class operators on $H(\Lambda)$, we have
\begeq
\label{eq32.001}
[P_{\eps},\chi_{j}]={\cal O}(h^{M}).
\endeq
The integer $M=M(N,\delta,\widetilde{\delta})>0$ is fixed and can be taken arbitrarily large by choosing the integer $N$ in Proposition \ref{prop2.1}
large enough. As observed and exploited in~\cite{HiSjVu}, such a choice of the cut-off $\chi_{j}$ is possible thanks to Proposition \ref{prop2.1}.

\begin{figure}
\centering
    \psfrag{1}{\Large{$\Re p_{\eps}$}}
    \psfrag{2}{\Large{$\Im \frac{p_{\eps}}{\eps}$}}
    \psfrag{3}{\Large{$E_2$}}
    \psfrag{4}{\Large{$E_4$}}
    \psfrag{5}{\Large{${\rm supp}\, \psi_{i,-}$}}
    \psfrag{6}{\Large{$\sim \eps^{\widetilde{\delta}}$}}
    \psfrag{7}{\Large{${\rm supp}\, \chi_3$}}
    \psfrag{8}{\Large{$\sim \eps^{\widetilde{\delta}}$}}
    \psfrag{9}{\Large{$F_3$}}
    \psfrag{10}{\Large{${\rm supp}\, \psi_{r,-}$}}
    \psfrag{11}{\Large{${\rm supp}\, \psi_{i,0}$}}
    \psfrag{12}{\Large{${\rm supp}\, \psi_{r,+}$}}
    \psfrag{13}{\Large{${\rm supp}\, \chi_1$}}
    \psfrag{14}{\Large{${\rm supp}\, \psi_{i,+}$}}
    \psfrag{15}{\Large{$F_1$}}
\scalebox{0.6} {\includegraphics{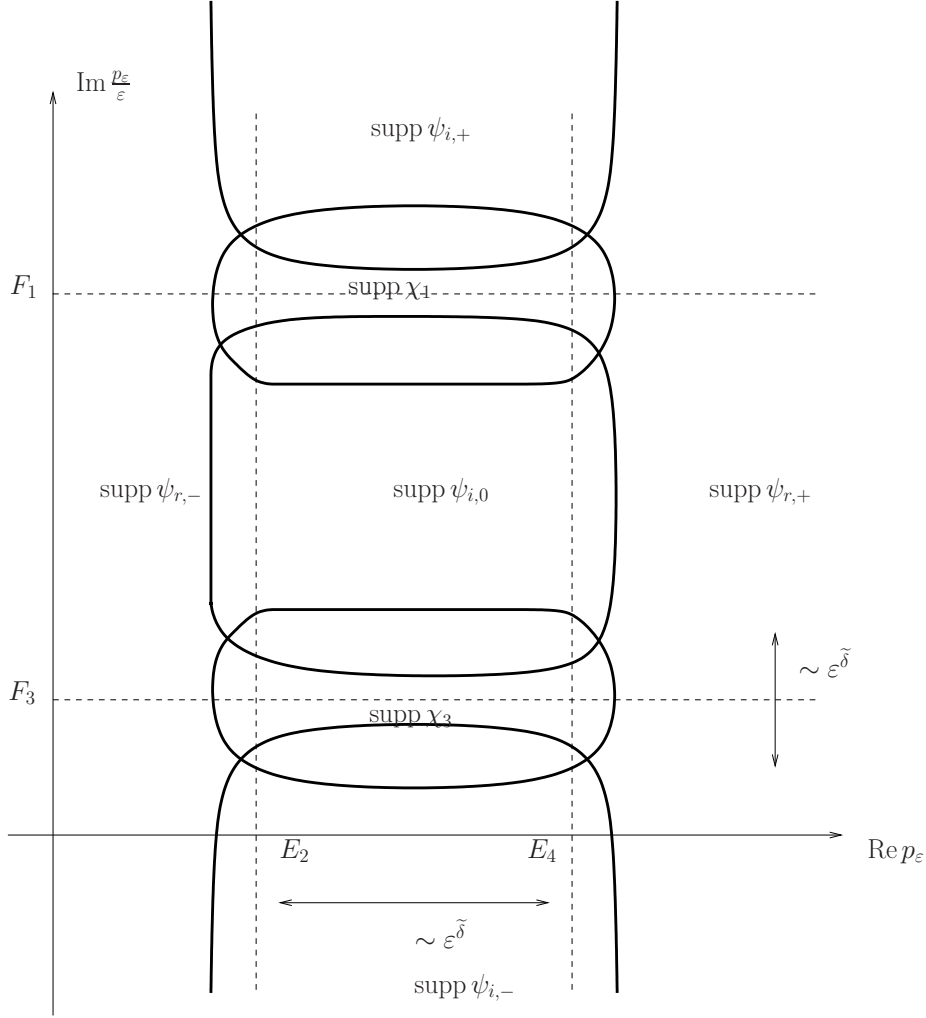}} \caption{A schematic representation of the partition of unity (\ref{eq32})
on the IR-manifold $\Lambda$, chosen according to Proposition \ref{prop2.1}. Here $p_{\eps}$ is the leading symbol of $P_{\eps}$,
acting on $H(\Lambda)$.}
\end{figure}

\bigskip
\noindent
The functions $0\leq \psi_{r,\pm}$ in (\ref{eq32}) are chosen so that
\begeq
\label{eq32.01}
\pm \Re P_{\eps}(\rho) \geq \frac{\eps^{\widetilde{\delta}}}{{\cal O}(1)},
\endeq
near the support of $\psi_{{r},\pm}$, respectively. Here we may assume that the support of $\psi_{r,+}$ contains a \neigh{} of infinity,
where $\Lambda$ agrees with $T^*M$, and there the estimate (\ref{eq32.01}) improves to the following one,
\begeq
\label{eq32.02}
\Re P_{\eps}(\rho) \geq \frac{m(\rho)}{{\cal O}(1)}.
\endeq

\medskip
\noindent
We shall now describe the properties of the cut-off functions $\psi_{i,\pm}$ and $\psi_{i,0}$, occurring in (\ref{eq32}). The function
$0\leq \psi_{i,-}\in C^{\infty}_0(\Lambda)$ is such that near the support of $\psi_{i,-}$ we have
$$
\Im P_{\eps} \leq \eps F_3 - \frac{\eps \eps^{\widetilde{\delta}}}{{\cal O}(1)},
$$
while near $\supp(\psi_{i,+})$ we have
$$
\Im P_{\eps} \geq \eps F_1 + \frac{\eps \eps^{\widetilde{\delta}}}{{\cal O}(1)}.
$$
Finally, we have
$$
\eps F_3 + \frac{\eps \eps^{\widetilde{\delta}}} {{\cal O}(1)} \leq \Im P_{\eps} \leq \eps F_1 - \frac{\eps \eps^{\widetilde{\delta}}}{{\cal O}(1)},
$$
near the support of $\psi_{i,0}\in C_0^{\infty}(\Lambda)$. Continuing to follow~\cite{HiSjVu}, we may and will arrange so that in the sense of
trace class operators on $H(\Lambda)$, we have
\begeq
\label{eq32.03}
A [P_{\eps},\psi_{i,{\bf \cdot}}]={\cal O}(h^{M}),\quad \cdot = \pm, 0,
\endeq
and also,
\begeq
\label{eq32.031}
[P_{\eps},\psi_{i,{\bf \cdot}}]A = {\cal O}(h^{M}),\quad \cdot = \pm, 0.
\endeq
Here $A$ is a microlocal cut-off to a region where $\abs{\Re P_{\eps}}\leq \eps^{\widetilde{\delta}}/{\cal O}(1)$, and
$M=M(N,\delta,\widetilde{\delta})$ is an integer having the same properties as the integer in (\ref{eq32.001}).

%We also arrange so that the functions $\psi_{\pm}$ and $\psi_0$ live away from the regions where
%$\abs{\Re P_{\eps} -E_j}\leq \frac{\eps^{3\delta}}{{\cal O}(1)}$, $j=2$, $4$, while the functions
%$\varphi_2$ and $\varphi_4$ are supported in the regions
%where $\abs{\Re P_{\eps}-E_j}\leq \eps^{3\delta}/{\cal O}(1)$, $j=2$, $4$, respectively.

\bigskip
Decomposing the operator $(z-P_{\eps})^{-1}$ according to (\ref{eq32}), we shall first analyze the trace of the integral
$$
\frac{1}{2\pi i}\int_{\gamma_1} \left(z-P_{\eps}\right)^{-1}\,\psi_{r,+}\,dz.
$$
We have
\begeq
\label{eq32.032}
(z-P_{\eps})^{-1} = {\cal O}\left(\frac{1}{\eps h^{N_0}}\right): H(\Lambda)\rightarrow H(\Lambda,m),
\endeq
and from~\cite{HiSjVu}, let us recall that
$$
\left(z-P_{\eps}\right)^{-1}\,\psi_{r,+} = {\cal O}\left(\frac{1}{\eps^{\widetilde{\delta}}}\right): H(\Lambda)\rightarrow H(\Lambda,m),
$$
provided that $h\leq {\cal O}(\eps^{9\widetilde{\delta}})$.
%Furthermore, the trace class norm of this operator acting on $H(\Lambda)$ is
%\begeq
%\label{eq32.04}
%\frac{1}{h^2}{\cal O}(\eps^{-\widetilde{\delta}}).
%\endeq
Thanks to the essentially elliptic estimate (\ref{eq32.01}), it is possible to construct a trace class parametrix for $z-P_{\eps}$, valid near
the support of $\psi_{r,+}$. We shall now describe briefly the main steps of this well known construction. Let
$\chi = \chi_{r,+}\in C^{\infty}_b(\Lambda)$ be such that $\Re P_{\eps} \geq \eps^{\widetilde{\delta}}/{\cal O}(1)$ near the support of $\chi$ and
$\chi=1$ near ${\rm supp}\, (\psi_{r,+})$. Let
$$
e_0(\rho,z,\eps)= \frac{\chi(\rho)}{z-P_{\eps}(\rho)}.
$$
Restricting the attention to a suitable bounded region of $\Lambda$, we see that
$$
\nabla^{\ell} e_0 = {\cal O}\left(\eps^{-\widetilde{\delta}} \eps^{-\widetilde{\delta}\ell}\right),\quad \ell \in \nat,
$$
while in a \neigh{} of infinity, where $\Lambda$ agrees with $\real^4$, we have
$$
\nabla^{\ell} e_0 = {\cal O}\left(\frac{1}{m}\right).
$$
It follows that, on the level of operators, we have
$$
(z-P_{\eps})e_0  = \chi + r_1, \quad r_1 = {\cal O}\left(\frac{h}{\eps^{2\widetilde{\delta}}}\right): H(\Lambda)\rightarrow H(\Lambda).
$$
Here and in what follows, we are quantizing the symbols on $\Lambda$ using the Weyl quantization on $\real^4$, as explained in Section \ref{review}.
Notice also that the trace class norm of $e_0$ does not exceed
$$
\frac{1}{h^2}{\cal O}\left(\log\frac{1}{\eps}\right).
$$
Furthermore, with
$$
e_1  = \frac{-r_1 \chi}{z-P_{\eps}},
$$
we get on the level of operators,
$$
(z-P_{\eps})\left(e_0+e_1\right) = \chi + r_1(1-\chi) + r_2,\quad
r_2 = {\cal O}\left(\frac{h^2}{\eps^{4\widetilde{\delta}}}\right): H(\Lambda)\rightarrow H(\Lambda).
$$
and continuing in this way we obtain the symbols $e_j$, $1\leq j \leq L$, and $r_k$, $1\leq k \leq L+1$, $L\in \nat$, such that on
the level of operators,
$$
(z-P_{\eps})\left(e_0+e_1+\ldots\, +e_L\right) = \chi + \sum_{k=1}^L r_k (1-\chi) + r_{L+1}.
$$
Here
$$
r_k = {\cal O}(h^{M_k}): H(\Lambda)\rightarrow H(\Lambda),
$$
where $M_k \rightarrow \infty$ as $k\rightarrow \infty$. It follows, using also (\ref{eq32.032}),
that modulo an expression whose trace class norm on $H(\Lambda)$ can be estimated by an arbitrarily high power of $h$, provided that we take
$L$ large enough, we have
$$
(z-P_{\eps})^{-1} \psi_{r,+} \equiv \left(e_0+e_1+\ldots\, +e_L\right)\psi_{r,+}.
$$
Estimating the trace class norm of $e_j \psi_{r,+}$, $j\geq 1$ and using that the length of $\gamma_1$ is ${\cal O}(\eps^{\widetilde{\delta}})$,
we obtain the following result.

\begin{prop}
\label{prop5.1}
We have
\begin{multline}
\label{eq32.044}
{\rm tr}\, \frac{1}{2\pi i} \int_{\gamma_1} \left(z-P_{\eps}\right)^{-1}\, \psi_{r,\pm}\,dz \\  =
\frac{1}{2\pi i} \frac{1}{(2\pi h)^2}\int_{\gamma_1} \int\!\!\!\int \frac{1}{z-p_{\eps}(\rho)}\, \psi_{r,\pm}(\rho)\, \mu(d\rho)\, dz +
{\cal O}\left(\eps^{\widetilde{\delta}} \frac{h \eps^{-2\widetilde{\delta}}}{h^2}\log\frac{1}{\eps}\right).
\end{multline}
\end{prop}

\bigskip
\noindent
We shall next consider the trace integral
\begeq
\label{eq32.05}
{\rm tr}\, \frac{1}{2\pi i} \int_{\gamma_1} \left(z-P_{\eps}\right)^{-1} \psi_{i,-}\, dz.
\endeq
Here we are no longer in the elliptic region, and in order to understand (\ref{eq32.05}), we shall proceed by means of a suitable trace class
perturbation, concentrated in the non-elliptic region, followed by a contour deformation argument. To be precise,
let $\widetilde{P}_{\eps}$ be such that $\widetilde{P}_{\eps}=P_{\eps}$ near $\supp(\psi_{i,-})$ and
with $\Im \widetilde{P}_{\eps} \leq \eps F_3 -\eps {\eps^{\widetilde{\delta}}}/{{\cal O}(1)}$ in the entire region where
$\abs{\Re P_{\eps}}< \eps^{\widetilde{\delta}}/{{\cal O}(1)}$, while $\widetilde{P}_{\eps}$ agrees with $P_{\eps}$ outside of a slightly larger
set of the form $\abs{\Re P_{\eps}}< \eps^{\widetilde{\delta}}/{{\cal O}(1)}$. We shall arrange, as we can, so that the operator $\widetilde{P}_{\eps}$ is of the form
\begeq
\label{eq32.051}
\widetilde{P}_{\eps} = P_{\eps}- i\eps \chi,
\endeq
where $\chi= \chi(\rho;h)\in C^{\infty}_0(\Lambda)$ satisfies
\begeq
\label{eq32.052}
\nabla^{\ell} \chi = {\cal O}_{\ell}(\eps^{-\ell\widetilde{\delta}}),\quad \ell \geq 0.
\endeq

\medskip
\noindent
We shall make use of the following result.

\begin{lemma}
\label{lemma5.2}
Assume that $\widetilde{\delta}>0$ is so small that $h\leq {\cal O}(\eps^{9\widetilde{\delta}})$, and recall that along $\gamma_1$,
the bound ${\rm dist}(z,{\rm Spec}(P_{\eps}))\geq \eps h^{N_0}$ holds, for some fixed $N_0\geq 1$. We have
\begin{enumerate}
\item In the region where $\Re z \in [E_2,E_4]$, $\eps F_1 \leq \Im z \leq {\cal O}(\eps^{\widetilde{\delta}})$, the operator
$$
z - \widetilde{P}_{\eps}: H(\Lambda,m)\rightarrow H(\Lambda)
$$
is bijective, with
\begeq
\label{eq32.06}
\left(z-\widetilde{P}_{\eps}\right)^{-1}  = \frac{{\cal O}(1)}{\Im z -\eps F_3 + \eps \eps^{\widetilde{\delta}}}: H(\Lambda) \rightarrow H(\Lambda,m).
\endeq
\item It holds, in the sense of trace class operators on $H(\Lambda)$,
\begeq
\label{eq32.1}
(z-\widetilde{P}_{\eps})^{-1}\psi_{i,-} - \left(z-{P}_{\eps}\right)^{-1}\psi_{i,-} ={\cal O}(h^{\widetilde{M}}),\quad z\in \gamma_1,
\endeq
where $\widetilde{M}$ can be made arbitrarily large, by taking the integer $M$ in {\rm (\ref{eq32.03})} large enough.
\end{enumerate}
\end{lemma}

\begin{proof}
In the region where $\abs{\Re P_{\eps}} < \eps^{\widetilde{\delta}}/{\cal O}(1)$, we have $\Im \widetilde{P}_{\eps} = {\cal O}(\eps)$ and
$$
\frac{1}{\eps} \Im \widetilde{P}_{\eps} \leq F_3 - \frac{\eps^{\widetilde{\delta}}}{{\cal O}(1)}.
$$
The statement (1) is obtained therefore by a standard application of the sharp G\aa{}rding inequality. See also Section 5 of~\cite{HiSjVu} for the
details of a similar argument.

When establishing (2), let us notice first that the expression in the left hand side of (\ref{eq32.1}) is equal to
\begeq
\label{eq32.2}
(z-\widetilde{P}_{\eps})^{-1}(\widetilde{P}_{\eps}-P_{\eps})\left(z-P_{\eps}\right)^{-1}\,\psi_{i,-}.
\endeq
Let $\widehat{\psi}_{i,-}$ be such that $\widehat{\psi}_{i,-}=1$ near $\supp(\psi_{i,-})$ and with $\widetilde{P}_{\eps}=P_{\eps}$
near $\supp(\widehat{\psi}_{i,-})$. We can also arrange that (\ref{eq32.03}) and (\ref{eq32.031}) are also valid for $\widehat{\psi}_{i,-}$.
Modulo ${\cal O}(h^{\widetilde{M}})$ in the trace class norm, we may replace the expression in
(\ref{eq32.2}) by
$$
\left(z-\widetilde{P}_{\eps}\right)^{-1}\left(\widetilde{P}_{\eps}-P_{\eps}\right)\left(1-\widehat{\psi}_{i,-}\right)
\left(z-P_{\eps}\right)^{-1}\,\psi_{i,-},
$$
whose trace class norm does not exceed ${\cal O}(\eps^{-\widetilde{\delta}})$ times the trace class norm of
$$
\left(1-\widehat{\psi}_{i,-}\right)\left(z-P_{\eps}\right)^{-1}\,\psi_{i,-}.
$$
When estimating the latter, we may follow some arguments of~\cite{Sj01}. When $L\in \nat$, let
$$
\psi_{i,-}:=\psi \prec \psi_1 \prec \ldots \prec \psi_L \prec \widehat{\psi}:=\widehat{\psi}_{i,-}.
$$
Here we arrange also that (\ref{eq32.03}) and (\ref{eq32.031}) are valid for $\psi_j$, $1\leq j \leq L$. Modulo terms that
are ${\cal O}(h^{\infty})$ in the trace class norm, we may write
\begin{multline}
\label{eq32.3}
\left(z-P_{\eps}\right)^{-1}\psi \\
\equiv \sum_{j=1}^L \psi_j (z-P_{\eps})^{-1}[P_{\eps},\psi_{j-1}]\ldots (z-P_{\eps})^{-1}[P_{\eps},\psi_1](z-P_{\eps})^{-1}\psi \\ \nonumber
+ (z-P_{\eps})^{-1}[P_{\eps},\psi_L] (z-P_{\eps})^{-1} [P_{\eps},\psi_{L-1}]\ldots\, [P_{\eps},\psi_1](z-P_{\eps})^{-1}\psi,
\end{multline}
and therefore, still modulo terms that are ${\cal O}(h^{\infty})$ in the trace class norm, we obtain that
\begin{multline}
\left(1-\widehat{\psi}\right)\left(z-P_{\eps}\right)^{-1}\psi \\ \nonumber
\equiv \left(1-\widehat{\psi}\right)
(z-P_{\eps})^{-1}[P_{\eps},\psi_L] (z-P_{\eps})^{-1} [P_{\eps},\psi_{L-1}]\ldots\, [P_{\eps},\psi_1](z-P_{\eps})^{-1}\psi.
\end{multline}
Here we recall that for $1\leq j\leq L$,
$$
A[P_{\eps},\psi_j] = {\cal O}(h^{\widetilde{M}}),\quad [P_{\eps},\psi_j]A={\cal O}(h^{\widetilde{M}}),
$$
in the trace class norm, where $A$ is a microlocal cut-off to the region where $\abs{\Re P_{\eps}}\leq \eps^{\widetilde{\delta}}/{\cal O}(1)$. Using
also that
$$
(z-P_{\eps})^{-1} (1-A)  = {\cal O}\left(\frac{1}{\eps^{\widetilde{\delta}}}\right): H(\Lambda) \rightarrow H(\Lambda,m),
$$
while the trace class norm of this operator is
$$
\frac{1}{h^2} {\cal O}\left(\frac{1}{\eps^{\widetilde{\delta}}}\right),
$$
as well as the fact that $[P_{\eps},\psi_j] = {\cal O}(h/\eps^{\widetilde{2\delta}})$, we obtain the second result.
\end{proof}

\bigskip
\noindent
An application of Lemma \ref{lemma5.2} shows that
\begeq
\label{eq32.5}
{\rm tr}\, \frac{1}{2\pi i} \int_{\gamma_1} \left(z-{P}_{\eps}\right)^{-1}\,\psi_{i,-}\,dz=
{\rm tr}\, \frac{1}{2\pi i}\int_{\gamma_1}\left(z-\widetilde{P}_{\eps}\right)^{-1}\,\psi_{i,-}\,dz + {\cal O}(h^{\widetilde{M}}).
\endeq
Here
\begeq
\label{eq33}
\int_{\gamma_1} \left(z-\widetilde{P}_{\eps}\right)^{-1}\,\psi_{i,-}\,dz=
\int_{\widetilde{\gamma}_1}\left(z-\widetilde{P}_{\eps}\right)^{-1}\,\psi_{i,-}\,dz,
\endeq
where $\widetilde{\gamma}_1$ is a piecewise linear contour contained in the region where $\Im z \geq \eps F_1$ and
having the same endpoints as $\gamma_1$, such that the closed contour $(-\gamma_1) \cup \widetilde{\gamma}_1$ is the positively oriented boundary
of the triangle with the third corner at the point $(E_2+E_4)/2+i\eps F_1+i\eps^{\widetilde{\delta}}$. When understanding
the integral in the right hand side of (\ref{eq33}), it will be convenient to decompose the function $\psi_{i,-}$ further, according to the values
taken by $\Re P_{\eps}$, so that we write
\begeq
\label{eq33.1}
\psi_{i,-} = \psi_{i,-,2}+\psi_{i,-,4} + \widetilde{\psi}_{i,-}.
\endeq
Here the functions $\psi_{i,-,j}$ are supported in the regions where $\abs{\Re P_{\eps}-E_j}\leq \eps^{3\widetilde{\delta}}/{\cal O}(1)$, $j=2,4$,
respectively, while the support of $\widetilde{\psi}_{i,-}$ stays away from these regions.

\medskip
\noindent
We shall consider first the integral
\begeq
\label{eq33.2}
\int_{\widetilde{\gamma}_1}\left(z-\widetilde{P}_{\eps}\right)^{-1}\,\widetilde{\psi}_{i,-}\,dz.
\endeq
In the support of $\widetilde{\psi}_{i,-}$, we have,
\begeq
\label{eq34}
\abs{z-\widetilde{P}_{\eps}(\rho)}\geq \frac{\eps^{3\widetilde{\delta}}}{{\cal O}(1)},\quad z\in \widetilde{\gamma}_1,
\endeq
which is an essentially elliptic estimate. Using (\ref{eq32.051}), (\ref{eq32.052}), and (\ref{eq34}) we then obtain that
$$
\nabla^{\ell} \left((z-\widetilde{P}_{\eps}(\rho))^{-1}\widetilde{\psi}_{i,-}\right) =
\eps^{-3\widetilde{\delta}} {\cal O}\left(\eps^{-3 \widetilde{\delta} \ell}\right),\quad \ell\in \nat.
$$
Also, the trace class norm of the corresponding Weyl quantization does not exceed
$$
\frac{1}{h^2} {\cal O}(\eps^{-2\widetilde{\delta}}).
$$
Here we may replace the function $\widetilde{\psi}_{i,-}$ by another cut-off with a slightly larger support. It is therefore clear that the
integral in (\ref{eq33.2}) can be understood by constructing an $h$-pseudo\-dif\-feren\-tial parametrix for
$z-\widetilde{P}_{\eps}$, valid near the support of $\widetilde{\psi}_{i,-}$, by following the same method as in the proof of Proposition \ref{prop4.1}.
We therefore conclude that the trace integral
$$
{\rm tr}\, \frac{1}{2\pi i} \int_{\widetilde{\gamma}_1} \left(z-\widetilde{P}_{\eps}\right)^{-1}\,\widetilde{\psi}_{i,-}\,dz
$$
is equal to
\begeq
\label{eq34.001}
\frac{1}{2\pi i}\frac{1}{(2\pi h)^2} \int_{\widetilde{\gamma}_1}\int\!\!\!\int \frac{1}{z-\widetilde{p}_{\eps}(\rho)}\widetilde{\psi}_{i,-}(\rho)\,\mu(d\rho)\,dz
+ {\cal O}\left(\frac{h \eps^{-7\widetilde{\delta}}}{h^2}\right).
\endeq
Here $\widetilde{p}_{\eps}$ is the leading symbol of $\widetilde{P}_{\eps}$, where we know that $\widetilde{p}_{\eps}-p_{\eps}={\cal O}(\eps)$ is supported
in a region where $\abs{\Re P_{\eps}}\leq \eps^{\widetilde{\delta}}/{\cal O}(1)$, away from the set where $\Im P_{\eps}\leq \eps F_3 -
\eps\eps^{\widetilde{\delta}}/{\cal O}(1)$.

\medskip
\noindent
According to (\ref{eq33}) and (\ref{eq33.1}), it remains to consider the integrals
\begeq
\label{eq34.01}
\int_{\widetilde{\gamma}_1} (z-\widetilde{P}_{\eps})^{-1}\,\psi_{i,-,j}\,dz,\quad j=2,4.
\endeq
Here an application of the bound (\ref{eq32.06}) in Lemma \ref{lemma5.2} together with the fact that the trace class norm of
$\psi_{i,-,j}$ on $H(\Lambda)$ for $j=2,4$ is
$$
{\cal O}\left(\frac{\eps^{3\widetilde{\delta}}}{h^2}\right),
$$
shows that the trace class norm of the expressions in (\ref{eq34.01}) does not exceed
\begeq
\label{eq34.02}
{\cal O}\left(\frac{\eps^{3\widetilde{\delta}}}{h^2}\right) \int_{\eps \eps^{\widetilde{\delta}}}^1 \frac{1}{s}\,ds =
{\cal O}(1) \frac{\eps^{3\widetilde{\delta}}}{h^2} \log\frac{1}{\eps}.
\endeq

\medskip
\noindent
Combining (\ref{eq32.5}), (\ref{eq33}), (\ref{eq33.1}), (\ref{eq34.001}), and (\ref{eq34.02}) we conclude that the trace
\begeq
\label{eq34.03}
{\rm tr}\,\frac{1}{2\pi i}\int_{\gamma_1} \left(z-P_{\eps}\right)^{-1}\,\psi_{i,-}\,dz
\endeq
is given by
\begeq
\label{eq34.04}
\frac{1}{2\pi i}\frac{1}{(2\pi h)^2}
\int_{\widetilde{\gamma}_1} \int\!\!\!\int \frac{1}{z-\widetilde{p}_{\eps}(\rho)}\widetilde{\psi}_{i,-}(\rho)\,\mu(d\rho)\,dz
+ {\cal O}(1) \frac{\eps^{3\widetilde{\delta}}}{h^2} \log\frac{1}{\eps},
\endeq
provided that the lower bound (\ref{eq31.01}) is strengthened to the following one,
\begeq
\label{eq34.041}
h\leq \eps^{10\widetilde{\delta}} \log\frac{1}{\eps}.
\endeq
Here the integral
$$
\int_{\widetilde{\gamma}_1} \int\!\!\!\int \frac{1}{z-\widetilde{p}_{\eps}(\rho)}\widetilde{\psi}_{i,-}(\rho)\,\mu(d\rho)\,dz
$$
can be replaced by the integral
$$
\int_{\gamma_1} \int\!\!\!\int  \frac{1}{z-p_{\eps}(\rho)}\psi_{i,-}(\rho)\,\mu(d\rho)\,dz,
$$
at the expense of an additional error not exceeding
$$
{\cal O}\left(\frac{\eps^{3\widetilde{\delta}}}{h^2}\right) \log \frac{1}{\eps}.
$$

\medskip
\noindent
The result obtained so far is typical of the behavior of the trace integrals in question in the non-elliptic region away from the tori $\widehat{\Lambda}_j$,
and so we state it in the following proposition.
\begin{prop}
\label{prop5.3}
Assume that $0<\widetilde{\delta}< 1$ is such that $h\leq \eps^{10\widetilde{\delta}}\log\frac{1}{\eps}$, and recall that
$$
\Im P_{\eps} \leq \eps F_3  - \frac{\eps \eps^{\widetilde{\delta}}}{{\cal O}(1)},
$$
near ${\rm supp}\,(\psi_{i,-})$, where $\psi_{i,-}$ satisfies {\rm (\ref{eq32.03})}, {\rm (\ref{eq32.031})}.
We have
\begin{multline}
\label{eq34.05}
{\rm tr}\,\frac{1}{2\pi i}\int_{\gamma_1} \left(z-P_{\eps}\right)^{-1}\,\psi_{i,-}\,dz \\
= \frac{1}{2\pi i}\frac{1}{(2\pi h)^2} \int_{\gamma_1} \int\!\!\!\int  \frac{1}{z-p_{\eps}(\rho)}\psi_{i,-}(\rho)\,\mu(d\rho)\,dz
+ {\cal O}\left(\frac{\eps^{3\widetilde{\delta}}}{h^2}\right) \log \frac{1}{\eps}.
\end{multline}
\end{prop}

\bigskip
\noindent
It is now easy to extend the result of Proposition \ref{prop5.3} to the trace integral
$$
{\rm tr}\,\frac{1}{2\pi i}\int_{\gamma_1} \left(z-P_{\eps}\right)^{-1}\,\left(1-\chi_1\right)\,dz.
$$
Indeed, the argument leading to the asymptotic result (\ref{eq34.05}) remains valid also when considering the trace integrals
$$
\frac{1}{2\pi i}\int_{\gamma_1}\left(z-P_{\eps}\right)^{-1}\,\psi_{i,0}\,dz
$$
and
$$
\frac{1}{2\pi i}\int_{\gamma_1}\left(z-P_{\eps}\right)^{-1}\,\chi_{3}\,dz,
$$
where we may also recall that $F_3 < F_1$.

\bigskip
\noindent
When considering the expression
$$
\int_{\gamma_1} \left(z-P_{\eps}\right)^{-1}\,\psi_{i,+}\,dz,
$$
we introduce a new trace class perturbation $\widehat{P}_{\eps}$ of $P_{\eps}$, such that
$\widehat{P}_{\eps}=P_{\eps}$ near $\supp(\psi_{i,+})$, and with
$$
\Im \widehat{P}_{\eps}\geq \eps F_1 +\frac{\eps \eps^{\widetilde{\delta}}}{{\cal O}(1)}
$$
in the entire region where
$\abs{\Re P_{\eps}}\leq \eps^{\widetilde{\delta}}/{\cal O}(1)$, and such that $\widehat{P}_{\eps}$ agrees with $P_{\eps}$ further away from this set.
The natural analogue of Lemma \ref{lemma5.2} continues to be valid for $(z-\widehat{P}_{\eps})^{-1}$, and when studying the trace of
$$
\int_{\gamma_1} (z-\widehat{P}_{\eps})^{-1}\, \psi_{i,+}\,dz,
$$
we can therefore deform the contour $\gamma_1$ downwards and introduce the decomposition similar to (\ref{eq33.1}),
so that the estimate (\ref{eq34}) holds for $z-\widehat{P}_{\eps}$ along the deformed contour, near the support of $\widetilde{\psi}_{i,+}$.

\bigskip
\noindent
The following is the main result of this subsection.
\begin{prop}
\label{prop5.4}
Assume that $E_2<0<E_4$, $\abs{E_j}\sim \eps^{\widetilde{\delta}}$, $j=2,4$, where $0<\widetilde{\delta}<1$ is so small that
$h\leq \eps^{10\widetilde{\delta}}\log(1/\eps)$. Let $\gamma_j$, $j=1,3$ be the horizontal segment given by $E_2 \leq \Re z \leq E_4$, $\Im z =\eps F_j$.
Let finally $0\leq \chi_j\in C^{\infty}_0(\Lambda)$ be a cut-off function to an $\eps^{\widetilde{\delta}}$-\neigh{} of $\widehat{\Lambda}_j$, $j=1,3$,
enjoying the commutator property {\rm (\ref{eq32.001})}. We have
\begin{multline}
{\rm tr}\,\frac{1}{2\pi i}\int_{\gamma_j} \left(z-P_{\eps}\right)^{-1}\,(1-\chi_j)\,dz \\ \nonumber
= \frac{1}{2\pi i}\frac{1}{(2\pi h)^2} \int_{\gamma_j} \int\!\!\!\int  \frac{1}{z-p_{\eps}(\rho)}(1-\chi_j(\rho))\,\mu(d\rho)\,dz
+ {\cal O}\left(\frac{\eps^{3\widetilde{\delta}}}{h^2}\right) \log \frac{1}{\eps},\quad j=1,3.
\end{multline}
\end{prop}

\subsection{The Birkhoff normal form and trace integrals near the tori}

In this subsection, we shall complete the trace analysis near the Diophantine levels by studying the integrals
$$
{\rm tr}\,\frac{1}{2\pi i}\int_{\gamma_j} \left(z-P_{\eps}\right)^{-1}\,\chi_j\,dz,
$$
say, when $j=1$.

\medskip
\noindent
When $z\in \gamma_1$, let us consider the equation
$$
\left(z-P_{\eps}\right)u=v,\quad u\in H(\Lambda,m),
$$
so that
$$
\left(z-P_{\eps}\right)\chi_1 u = \chi_1 v + [\chi_1,P_{\eps}]u.
$$
Applying the operator $U:=U_1$ of Proposition \ref{prop2.1}, we get
$$
\left(z-P^{(N)}_1-R_{N+1,1}\right)U \chi_1  u = U \chi_1 v + T_N u,
$$
with the trace class norm of $T_N$ on $H(\Lambda)$ being ${\cal O}(h^M)$, where $M$ can be taken arbitrarily large, by taking the integer $N$ in
Proposition \ref{prop2.1} large enough. Using the fact that $R_{N+1,1}(x,\xi,\eps;h) = {\cal O}\left((h,\eps,\xi)^{N+1}\right)$ and modifying $T_N$
slightly, we get
$$
\left(z-P^{(N)}_1\right) U \chi_1 u = U\chi_1 v + T_N u.
$$
It is therefore clear, in view of the fact that $\eps \geq h^K$, $K\geq 1$ fixed and since the bound (\ref{eq31.1}) holds, that
\begeq
\label{eq34.11}
\chi_1 \left(z-P_{\eps}\right)^{-1} = V \left(z-P^{(N)}_1\right)^{-1}U \chi_1 + T_{1,N},
\endeq
where $T_{1,N}$ has the same trace class norm bound as $T_N$. Here $V$ is a microlocal inverse of $U$. Therefore,
\begin{multline}
\label{eq34.12}
{\rm tr}\,\frac{1}{2\pi i}\int_{\gamma_1} \left(z-P_{\eps}\right)^{-1}\,\chi_1\,dz \\
= {\rm tr}\,\frac{1}{2\pi i}\int_{\gamma_1} \left(z-P^{(N)}_1\right)^{-1}\,U\chi_1 V \,dz +{\cal O}(h^M).
\end{multline}
Here we may recall, following~\cite{HiSjVu}, that $U \chi_1 V  = \chi(hD_x/\eps^{\widetilde{\delta}})$, where $\chi\in C^{\infty}_0(\real^2)$ is a
standard cut-off to a \neigh{} of $\xi=0$. Let us write in what follows $\chi_{\eps}(\xi) = \chi(\xi/\eps^{\widetilde{\delta}})$.

\medskip
\noindent
Now the eigenvalues of the translation invariant operator $P^{(N)}_1(hD_x,\eps;h)$, acting on
$L^2_{\theta}({\bf T}^2)$ are given by
\begeq
\label{eq34.13}
\mu(k) := P^{(N)}_1 \left(h\left(k-\frac{k_0}{4}\right)-\frac{S}{2\pi}\right),\quad k\in \z^2,
\endeq
and by (\ref{eq34.12}), we conclude that
\begin{multline}
\label{eq34.2}
{\rm tr}\, \frac{1}{2\pi i}\int_{\gamma_1} (z-P_{\eps})^{-1}\chi_{1}\,dz  \\
= \sum_{k\in {\bf Z}^2}\frac{1}{2\pi i}
\int_{\gamma_1} (z-\mu(k))^{-1} \chi_{\eps}\left(h(k-\frac{k_0}{4})-\frac{S}{2\pi}\right)\,dz
+ {\cal O}(h^M).
\end{multline}

\medskip
\noindent
We have
\begeq
\label{eq34.21}
\mu(k) = \widehat{p}\left(h\left(k-\frac{k_0}{4}\right)-\frac{S}{2\pi},\eps\right) + {\cal O}(h),
\endeq
where
$$
\widehat{p}(\xi,\eps) = p(\xi) + i \eps\langle{q}\rangle(\xi) +{\cal O}(\eps^2)
$$
is the leading symbol of $P^{(N)}_1(hD_x,\eps;h)$. We would like to compare the expression in the right hand side of (\ref{eq34.2}) with the integral
\begeq
\label{eq34.3}
\frac{1}{2\pi i}\frac{1}{(2\pi h)^2}\int_{\gamma_1}\,\int\!\!\!\int \left(z-\widehat{p}(\xi,\eps)\right)^{-1}
\chi_{\eps}(\xi)\,dx\,d\xi\,dz,
\endeq
the integration in the $(x,\xi)$ variables being carried out over $T^*{\bf T}^2$.

\medskip
\noindent
Integrating out the $x$-variables in (\ref{eq34.3}), we shall first compare the expressions
$$
\int_{{\bf R}^2}\int_{\gamma_1} \left(z-\widehat{p}(\xi,\eps)\right)^{-1}\, \chi_{\eps}(\xi)\,d\xi\,dz
$$
and
$$
\sum_{k\in {\bf Z}^2} h^2 \int_{\gamma_1}
\left(z-\widehat{p}\left(h\left(k-\frac{k_0}{4}\right)-\frac{S}{2\pi},\eps\right)\right)^{-1}
\chi_{\eps}\left(h\left(k-\frac{k_0}{4}\right)-\frac{S}{2\pi}\right)\,dz.
$$
When $\xi\in \real^2$ is such that
$$
\xi\in h\left(k-\frac{k_0}{4}\right) -\frac{S}{2\pi}+ [0,h)^2,
$$
for some $k\in \z^2$, let us write
$$
[\xi]=h\left(k -\frac{k_0}{4}\right)-\frac{S}{2\pi}.
$$
Let us consider
\begeq
\label{eq35}
\int_{{\bf R}^2}\int_{\gamma_1} \left(\left(z-\widehat{p}(\xi,\eps)\right)^{-1} -
 \left(z-\widehat{p}([\xi],\eps)\right)^{-1}\right)
\chi_{\eps}(\xi)\,d\xi\,dz.
\endeq
Let $a$, $b$ be the endpoints of $\gamma_1$. Then with suitable branches of the logarithm, we have
\begin{multline}
\label{eq35.0}
\int_{\gamma_1} \left(\left(z-\widehat{p}(\xi,\eps)\right)^{-1} -
 \left(z-\widehat{p}([\xi],\eps)\right)^{-1}\right)\,dz \\
= \left(\log \left(b-\widehat{p}(\xi,\eps)\right)-\log \left(b-\widehat{p}([\xi],\eps)\right)\right) \\
- \left(\log \left(a-\widehat{p}(\xi,\eps)\right)-\log \left(a-\widehat{p}([\xi],\eps)\right)\right).
\end{multline}
In general, for $z$, $w\in \comp$, we have
$$
\log z - \log w  = \int_{z}^w \frac{1}{\zeta}\,d\zeta,
$$
where the choice of curve joining $w$ and $z$ depends on the choices of branches of $\log z$, $\log w$. If we have the same branch then
\begeq
\label{eq35.1}
\abs{\log z -\log w}\leq \frac{C_0 \abs{z-w}}{{\rm min}(\abs{z},\abs{w})}.
\endeq
If the branch cut passes between $z$ and $w$, we have to add a constant. In the case of (\ref{eq35.0}), this happens precisely when
$\widehat{p}(\xi,\eps)$ and $\widehat{p}([\xi],\eps)$ are on the opposite sides of $\gamma_1$. Now let us concentrate on one of the terms in
(\ref{eq35.0}), say
\begeq
\label{eq35.2}
\log \left(a-\widehat{p}(\xi,\eps)\right)  - \log \left(a-\widehat{p}([\xi],\eps)\right).
\endeq
If
\begeq
\label{eq35.3}
\abs{\Re a - \Re \widehat{p}(\xi,\eps)}\leq C_0 h,
\endeq
for a suitable fixed constant $C_0>0$, we estimate the two terms separately and get that the contribution to (\ref{eq35}) in this case is
$$
\int_{E(C_0)} \chi_{\eps}(\xi) \log \left(a-\widehat{p}(\xi,\eps)\right)\,d\xi \leq {\cal O}(1) \int_0^{C_1 h}-\log t\, dt = {\cal O}(1)
h \log \frac{1}{h}.
$$
Here $E(C_0)\subset \real^2$ is the set of all $\xi\in \real^2$ such that (\ref{eq35.3}) holds. If we assume that $a$ has been chosen so that for all
$\xi\in \real^2$,
\begeq
\label{eq35.4}
\abs{a-\widehat{p}([\xi],\eps)}\geq \frac{h}{{\cal O}(1)},
\endeq
then we get the same estimate for
$$
\int \chi_{\eps}(\xi) \log \left(a-\widehat{p}([\xi],\eps)\right)\,d\xi.
$$
In the region where $\abs{\Re a - \Re \widehat{p}(\xi,\eps)}\geq C_0 h$, let us first assume that we have the same branches of
$\log (a-\widehat{p}(\xi,\eps))$ and $\log(a-\widehat{p}([\xi],\eps))$. Then by (\ref{eq35.1}),
$$
\log (a-\widehat{p}(\xi,\eps)) - \log(a-\widehat{p}([\xi],\eps)) = {\cal O}(1) \frac{h}{\Re a- \Re \widehat{p}(\xi,\eps)},
$$
and the corresponding contribution to (\ref{eq35}) is
$$
{\cal O}(h) \int_h^1 \frac{1}{t}\,dt  = {\cal O}\left(h \log \frac{1}{h}\right).
$$
It remains to estimate the integral of the extra contributions $\pm 2\pi i$, to $\log (a-\widehat{p}(\xi,\eps)) - \log(a-\widehat{p}([\xi],\eps))$
from points $\xi$ for which $\widehat{p}(\xi,\eps)$ and $\widehat{p}([\xi],\eps)$ are on the opposite sides of $\gamma_1$. But the Lebesgue measure of the
set of such points is ${\cal O}(h)$, so the corresponding contribution to the integral is ${\cal O}(h)$.

\medskip
\noindent
Summing up our estimates and computations, we see that the expression (\ref{eq35}) is ${\cal O}(h\log \frac{1}{h})$. Arguing similarly and using
(\ref{eq34.21}), we obtain that
\begeq
\label{eq35.5}
\int_{{\bf R}^2}\int_{\gamma_1} \left(\left(z-\widehat{p}([\xi],\eps)\right)^{-1} -
 \left(z-\mu(k)\right)^{-1}\right) \chi_{\eps}([\xi])\,d\xi\,dz = {\cal O}(h\log \frac{1}{h}).
\endeq
Finally, we find that also,
\begeq
\label{eq35.6}
\int_{{\bf R}^2} \int_{\gamma_1} \left(z-\widehat{p}([\xi],\eps)\right)^{-1} \left(\chi_{\eps}(\xi)-\chi_{\eps}([\xi])\right)\,d\xi\,dz =
{\cal O}(h\log \frac{1}{h}).
\endeq

\bigskip
\noindent
We summarize the result of this subsection in the following proposition. Here we also use that the
integral over $T^*{\bf T}^2$ in (\ref{eq34.3}) can be transformed into the corresponding integral over
$\Lambda$ by means of the canonical transformation associated to the operator $U$.
\begin{prop}
\label{prop5.5}
Assume that $E_2<0<E_4$, $\abs{E_j}\sim \eps^{\widetilde{\delta}}$, $j=2,4$, where $0<\widetilde{\delta}<1$ is so small that
$h\leq \eps^{10\widetilde{\delta}}\log(1/\eps)$. Let $\gamma_j$, $j=1,3$ be the horizontal segment given by $E_2 \leq \Re z \leq E_4$, $\Im z =\eps F_j$.
Let finally $0\leq \chi_j\in C^{\infty}_0(\Lambda)$ be a cut-off function to an $\eps^{\widetilde{\delta}}$-\neigh{} of $\widehat{\Lambda}_j$, $j=1,3$,
enjoying the commutator property {\rm (\ref{eq32.001})}. We have
\begin{multline}
{\rm tr}\,\frac{1}{2\pi i}\int_{\gamma_j} \left(z-P_{\eps}\right)^{-1}\,\chi_j\,dz \\ \nonumber
= \frac{1}{2\pi i}\frac{1}{(2\pi h)^2} \int_{\gamma_1} \int\!\!\!\int  \frac{1}{z-p_{\eps}(\rho)}\chi_j(\rho)\,\mu(d\rho)\,dz
+ {\cal O}\left(\frac{1}{h}\right) \log \frac{1}{\eps},\quad j=1,3.
\end{multline}
\end{prop}

\subsection{End of the proof}
Combining Propositions \ref{prop4.2}, \ref{prop5.4}, and \ref{prop5.5}, we obtain that the number of eigenvalues of $P_{\eps}$ in the
rectangle $R$ in (\ref{eq25.2}) is equal to
\begin{multline}
\label{eq35.7}
\Re \frac{1}{2\pi i} \frac{1}{(2\pi h)^2}\int_{\gamma} \int\!\!\!\int \frac{1}{z-p_{\eps}(\rho)}\, \mu(d\rho)\,dz +
\frac{1}{h^2} {\cal O}\left(\eps^{3\widetilde{\delta}} \log{\frac{1}{\eps}}\right) \\
= \frac{1}{(2\pi h)^2}\int\!\!\!\int_{\Omega(\eps,R)} \mu(d\rho)
+ \frac{1}{h^2}{\cal O}\left(\eps^{3\widetilde{\delta}} \log{\frac{1}{\eps}}\right).
\end{multline}
Here $\Omega(\eps,R) = p_{\eps}^{-1}(R)\subset \Lambda$ and $p_{\eps}\in C^{\infty}(\Lambda)$ is the leading symbol of $P_{\eps}$,
acting on $H(\Lambda)$.

\medskip
\noindent
Let us now recall the $C^{\infty}$ canonical transformation
$$
\kappa: {\rm neigh}(p^{-1}(0),T^*M)\rightarrow {\rm neigh}(p^{-1}(0),\Lambda),
$$
introduced in (\ref{eq27.12.2}), so that $\kappa$ is ${\cal O}(\eps)$--close to the identity in the $C^{\infty}$--sense, for each fixed $T\geq T_0$,
$T_0>0$ large enough. An application of (\ref{eq27.12.3}) shows that
\begeq
\label{eq35.8}
p_{\eps}(\kappa(\rho)) = p(\rho) + i\eps \left(q-H_p G\right)(\rho)+{\cal O}(\eps^2),\quad \rho\in T^*M.
\endeq
In the compact case, we obtain the same expression for the transformed symbol.

\medskip
\noindent
It follows from (\ref{eq35.8}) and the properties of the function $G$, recalled in Section \ref{review}, that the set $\Omega(\eps,R)\subset \Lambda$
is ${\cal O}(\eps)$--close to the set
$$
\Omega([E_2,E_4])  := \bigcup_{E_2 \leq E \leq E_4} \Omega(E)\subset T^*M,
$$
introduced in (\ref{eq26}).  Therefore,
\begeq
\label{eq35.9}
\int\!\!\!\int_{\Omega(\eps,R)} \mu(d\rho) = \int\!\!\!\int_{\Omega([E_2,E_4])} dx\,d\xi + {\cal O}(\eps).
\endeq
Here $T>0$ is large enough fixed. Combining (\ref{eq35.7}) and (\ref{eq35.9}), we complete the proof of Theorem \ref{theo_main}.
%We now return to (\ref{eq35.7}) and observe that
%$$
%\mu(d\rho) = \frac{\sigma^2}{2!}\biggl|_{\Lambda} = \frac{1}{2} d\left(\omega\wedge \sigma\right)\biggl|_{\Lambda}.
%$$
%Here $\sigma$ is the complex symplectic $(2,0)$-form on $T^*\widetilde{M}$ and $\omega$ is the fundamental $(1,0)$-form. An application of Stokes'
%formula then gives that uniformly in $T\geq T_0$, $T_0$ large enough,
%\begin{multline}
%\label{eq35.9}
%\int\!\!\!\int_{\Omega(\eps,R)} \mu(d\rho) = \frac{1}{2}\int_{\partial \Omega(\eps,R)} \omega \wedge \sigma =
%\frac{1}{2} \int_{\partial \Omega([E_2,E_4])} \omega \wedge \sigma + {\cal O}(\eps) \\
% =  \int\!\!\!\int_{\Omega([E_2,E_4])} dx\,d\xi + {\cal O}(\eps).
%\end{multline}

%\begeq
%\label{eq35.10}
%\frac{1}{(2\pi h)^2}\int\!\!\!\int_{E_2\leq p\leq E_4} 1_{[F_3,F_1]}\left((q-H_p G)(x,\xi))\right)\,dx\,d\xi
%+ \frac{1}{h^2}{\cal O}\left(\eps^{3\widetilde{\delta}} \log{\frac{1}{\eps}}\right).
%\endeq
%Here the integral is equal to
%$$
%\int_{E_2}^{E_4} \left(\int_{p^{-1}(E)} 1_{[F_3,F_1]}\left(q-H_p G\right)\,{\cal L}_E(d(x,\xi))\right)\,dE,
%$$
%where ${\cal L}_E$ is the Liouville measure on $p^{-1}(E)$. As $T\rightarrow \infty$, we have the almost everywhere convergence along $p^{-1}(E)$,
%$$
%1_{[F_3,F_1]}(q-H_p G)\rightarrow 1_{\Omega(E)},
%$$
%with the set $\Omega(E)\subset p^{-1}(E) \cap T^*M$ introduced in the statement of Theorem \ref{theo_main}. Hence by the dominated convergence,
%we obtain that for $E\in [E_2,E_4]$,
%$$
%\int_{p^{-1}(E)} 1_{[F_3,F_1]} \left(q-H_p G\right)\,{\cal L}_E(d(x,\xi)) \rightarrow \int_{p^{-1}(E)} 1_{\Omega(E)}\,{\cal L}_E(d(x,\xi)).
%$$

\section{Weyl asymptotics in the periodic case}
\label{periodic}
\setcounter{equation}{0}
In this section, we shall explain how the results and methods of the work~\cite{HiSj04}, combined with the methods
of the present work, can be used to obtain an analog of Theorem \ref{theo_main} in the case when instead of the complete integrability assumptions, we assume that Hamilton flow of $p$ is periodic. It turns out that the analysis of the periodic case will proceed in full analogy with the previously analyzed completely integrable case. The following discussion will therefore be somewhat brief.

\medskip
\noindent
In order to fix the ideas, throughout this section, we shall consider the case when $M=\real^2$. Let $P_{\eps}$ be an operator satisfying all the assumptions made in
subsection 2.1, and in particular, (\ref{eq9.2}). Assume that for $E\in {\rm neigh}(0,\real)$, the following condition holds,
\begin{eqnarray}
\label{eq36}
\hbox{The }H_p\hbox{-flow is periodic on }p^{-1}(E)\cap \real^4 \hbox{
with}\\ \hbox{period }T(E)>0 \hbox{ depending analytically on }E.\nonumber
\end{eqnarray}

\medskip
\noindent
As in (\ref{eq15}), we set
\begeq
\label{eq37}
\langle{q}\rangle  = \frac{1}{T(E)}\int_{0}^{T(E)} q\circ \exp(tH_p)\,dt\quad \wrtext{on}\,\, p^{-1}(E)\cap \real^4,
\endeq
and notice that the functions $p$ and $\langle{q}\rangle$ are in involution, so that $H_p \langle{q}\rangle=0$. Similarly to (\ref{eq18}), it
is established in~\cite{HiSj04} that,
\begeq
\label{eq38}
\frac{1}{\eps}\Im \left({\rm Spec}(P_{\eps})\cap \{z; \abs{\Re z}\leq \delta\}\right)
\subset \left[\min_{p^{-1}(0)\cap {\bf R}^4} \langle{q}\rangle-o(1),
\max_{p^{-1}(0)\cap {\bf R}^4}\langle{q}\rangle+o(1)\right],
\endeq
as $\eps$, $h$, $\delta\rightarrow 0$.

\medskip
\noindent
Let
$$
F_j \in [\min_{p^{-1}(0)\cap {\bf R}^4} \langle{q}\rangle, \max_{p^{-1}(0)\cap {\bf R}^4} \langle{q}\rangle],\quad j=1,3,\quad F_3 < F_1,
$$
and let us introduce the associated level sets
\begeq
\label{eq39}
\Lambda_{0,F_j} = \{\rho\in \real^4;\, p(\rho)=0,\,\,\langle{q}\rangle(\rho) = F_j\},\quad j=1,3.
\endeq
As in~\cite{HiSj04}, we shall work under the general assumption that for $j=1,3$,
\begeq
\label{eq40}
T(0)\hbox{ is the minimal period of every }H_p\hbox{-trajectory in
}\Lambda _{0,F_j}.
\endeq
We shall furthermore assume that
\begeq
\label{eq41}
dp,\, d \langle q\rangle \hbox{ are linearly independent at every
point of }\Lambda _{0,F_j},\quad j=1,3.
\endeq
It follows that each connected component of the level set $\Lambda_{0,F_j}$ is a two-dimensional Lagrangian torus. For simplicity, we shall assume that
the sets $\Lambda_{0,F_j}$ are both connected, $j=1,3$. We can then make a real analytic canonical transformation
$$
\kappa_j: {\rm neigh}(\Lambda_{0,F_j},\real^4)\rightarrow {\rm neigh}(\xi=0,T^*{\bf T}^2),
$$
given by the action-angle coordinates near $\Lambda_{0,F_j}$, so that when expressed in terms of the coordinates $x$ and $\xi$ on the torus side, we have
$p\circ \kappa_j^{-1} = p_j(\xi_1)$, $\langle{q\rangle}\circ \kappa_j^{-1} =  \langle{q_j}\rangle(\xi)$.

\medskip
\noindent
Following the analysis carried out in Section 4 of~\cite{HiSj04}, we shall now state the following result, which can be viewed as an analog of
Proposition \ref{prop2.1} in the present periodic situation.
\begin{prop}
\label{prop6.1}
Let us keep all the general assumptions of subsection {\rm 2.1} and make furthermore the assumptions
{\rm (\ref{eq36})}, {\rm (\ref{eq40})}, and {\rm (\ref{eq41})}. Assume that $\eps = {\cal O}(h^{\delta})$, $0<\delta\leq 1$, is such $h/\eps \leq \delta_0$,
for some $0<\delta_0$ sufficiently small. There exists a smooth IR-manifold $\Lambda \subset \comp^4$ and smooth Lagrangian tori $\widehat{\Lambda}_1$ and
$\widehat{\Lambda}_3\subset \Lambda$, such that when $\rho\in \Lambda$ is away from a small but fixed \neigh{} of
$\widehat{\Lambda}_1\cup \widehat{\Lambda}_3$ in $\Lambda$, we have
\begeq
\label{eq42}
\abs{\Re P_{\eps}(\rho)}\geq \frac{1}{{\cal O}(1)}
\endeq
or
\begeq
\label{eq43}
\abs{\Im P_{\eps}(\rho)-\eps F_1}\geq \frac{\eps}{{\cal O}(1)}\quad \wrtext{and}\quad
\abs{\Im P_{\eps}(\rho)-\eps F_3}\geq \frac{\eps}{{\cal O}(1)}.
\endeq
The manifold $\Lambda$ is an ${\cal O}(\eps+\frac{h}{\eps})$-perturbation of $\real^4$ in the $C^{\infty}$--sense, and it agrees with $\real^4$ outside
of a \neigh{} of $p^{-1}(0)\cap \real^4$. We have
$$
P_{\eps}  = {\cal O}(1): H(\Lambda,m)\rightarrow H(\Lambda).
$$
When $j=1,3$, there exists an elliptic $h$--Fourier integral operator
$$
U_j = {\cal O}(1): H(\Lambda)\rightarrow L^2_{\theta}({\bf T}^2),
$$
such that microlocally near $\widehat{\Lambda}_j$, we have
$$
U_j P_{\eps}  = \widehat{P}_j U_j.
$$
Here $\widehat{P}_j=\widehat{P}_j(hD_x,\eps,\frac{h}{\eps};h)$ is an operator acting on $L^2_{\theta}({\bf T}^2)$ with the symbol
$$
\widehat{P}_j\left(\xi,\eps,\frac{h}{\eps};h\right) \sim p_j(\xi_1) + \eps \sum_{k=0}^{\infty} h^k r_{j,k} \left(\xi,\eps,\frac{h}{\eps}\right),\quad
\abs{\xi}\leq \frac{1}{{\cal O}(1)},
$$
where
$$
r_{j,0}(\xi) = i\langle{q_j}\rangle(\xi) + {\cal O}\left(\eps + \frac{h}{\eps}\right),
$$
and
$$
r_{j,k}(\xi) = {\cal O}\left(\eps+\frac{h}{\eps}\right),\quad k\geq 1.
$$
\end{prop}

\medskip
\noindent
From~\cite{HiSj04}, we also infer that the natural analog of Proposition \ref{prop2.2} is valid, when the spectral parameter $z$ belongs to the
rectangle
$$
\left[-\frac{1}{C},\frac{1}{C}\right] + i\eps \left[F_j -\frac{1}{C},F_j + \frac{1}{C}\right],\quad j=1,3,
$$
for a sufficiently large constant $C>0$, and is such that ${\rm dist}(z,{\rm Spec}(P_{\eps}))\geq \eps h^{N_0}$, for some $N_0\geq 1$ fixed.

\medskip
\noindent
Let
\begeq
\label{eq44}
R  = (E_2,E_4) + i\eps (F_3,F_1),
\endeq
where $E_2<0<E_4$ are independent of $h$, with $\abs{E_j}<1/C$, $j=2,4$, for $C>0$ sufficiently large. We decompose $\partial R$ as in (\ref{eq26.1}),
and notice that the analysis of Section \ref{vertical} applies to the traces of the integrals
$$
\frac{1}{2\pi i}\int_{\gamma_j} (z-P_{\eps})^{-1}\,dz,\quad j=2,4,
$$
as it stands, so that the statement of Proposition \ref{prop4.2} continues to hold in the present situation. For future reference, we state it as the
following result. In its formulation, the notation $\mu(d\rho)$ stands for the symplectic volume element on $\Lambda$, and $p_{\eps}$ is the leading
symbol of $P_{\eps}$, viewed as an unbounded operator on $H(\Lambda)$.

\begin{prop}
\label{prop6.2}
Let $E_2 < 0 < E_4$ be independent of $h$ and such that $\abs{E_j} < 1/C$, for $C>0$ sufficiently large, $j=1,2$. Assume that
$ \eps ={\cal O}(h^{\delta})$, $0<\delta\leq 1$, is such that $h/\eps \leq
\delta_0$, for some $\delta_0>0$ small enough. Assume also that $0<\widetilde{\delta}<1$ is such that $h\leq \eps^{9\widetilde{\delta}}$.
When $\gamma_j$ is the vertical segment given by $\Re z = E_j$, $\eps F_3 \leq \Im z \leq \eps F_1$, we have, for $j=2,4$,
\begin{multline*}
\Re {\rm tr}\, \frac{1}{2\pi i}\int_{\gamma_j} \left(z-P_{\eps}\right)^{-1}\,dz  \\ =
\Re \frac{1}{2\pi i} \frac{1}{(2\pi h)^2}\int_{\gamma_j} \int\!\!\!\int \frac{1}{z-p_{\eps}(\rho)}\, \mu(d\rho)\,dz +
\frac{1}{h^2} {\cal O}\left(\eps^{3\widetilde{\delta}} \log{\frac{1}{\eps}}\right).
\end{multline*}
\end{prop}

\bigskip
\noindent
When analyzing the trace integral
$$
{\rm tr}\, \frac{1}{2\pi i}\int_{\gamma_j} (z-P_{\eps})^{-1}\,dz,\quad j=1,3,
$$
say, with $j=1$, we use Proposition \ref{prop6.1} to introduce the following smooth partition of unity on the manifold $\Lambda$, similar to
(\ref{eq32}),
\begeq
\label{eq45}
1 = \chi_1 + \chi_3 + \psi_{r,+} + \psi_{r,-} + \psi_{i,-} + \psi_{i,0} + \psi_{i,+}.
\endeq
Here $0\leq \chi_j \in C^{\infty}_0(\Lambda)$ is supported in a small flow-invariant \neigh{} of $\widehat{\Lambda}_j$, where $P_{\eps}$ is
intertwined with $\widehat{P}_j$, according to Proposition \ref{prop6.1}, and $\chi=1$ near $\widehat{\Lambda}_j$, $j=1,3$. As in~\cite{HiSj05}, we
choose $\chi_j$ so that in the sense of trace class operators on $H(\Lambda)$, we have
\begeq
\label{eq46}
[P_{\eps},\chi_j] = {\cal O}(h^{\infty}).
\endeq

\medskip
\noindent
The functions $0\leq \psi_{r,\pm}\in C^{\infty}_{b}(\Lambda)$ are such that
\begeq
\label{eq47}
\pm \Re P_{\eps} (\rho) > \frac{1}{C},
\endeq
near the support of $\psi_{r,\pm}$, respectively. Here, as in Section \ref{horizontal}, we may assume that the support of $\psi_{r,+}$ is
unbounded, and near infinity, the bound (\ref{eq47}) improves to
\begeq
\label{eq48}
\Re P_{\eps}(\rho) \geq \frac{m(\rho)}{{\cal O}(1)}.
\endeq
We now come to describe the properties of the functions $\psi_{i,\pm}$ and $\psi_{i,0}$ in (\ref{eq45}). These non-negative
functions in $C^{\infty}_0(\Lambda)$  are supported in regions invariant under the $H_p$--flow, and the estimate
$$
\Im P_{\eps} \leq \eps F_3 - \frac{\eps}{C}
$$
holds near ${\rm supp}\,\psi_{i,-}$. Similarly, near ${\rm supp}\, \psi_{i,+}$ we have
$$
\Im P_{\eps} \geq \eps F_1 + \frac{\eps}{C},
$$
and finally, the bound
$$
\eps F_3 + \frac{\eps}{C} \leq \Im P_{\eps} \leq \eps F_1 -\frac{\eps}{C}
$$
is valid in a \neigh{} of the support of $\psi_{i,0}$. As in Section \ref{horizontal}, we may and will arrange so that in the sense of trace class
operators on $H(\Lambda)$, we have
\begeq
\label{eq49}
A[P_{\eps},\psi_{i,\cdot}] = {\cal O}(h^{\infty}),\quad [P_{\eps},\psi_{i,\cdot}]A = {\cal O}(h^{\infty}),\quad \cdot  = \pm,0.
\endeq
Here $A$ is a microlocal cut-off to a region where $\abs{\Re P_{\eps}}<1/{\cal O}(1)$.

\medskip
\noindent
The analysis of the trace integrals
\begeq
\label{eq50}
{\rm tr}\,\frac{1}{2\pi i}\int_{\gamma_1} (z-P_{\eps})^{-1}\psi_{r,\pm}\,dz
\endeq
proceeds exactly as in the proof of Proposition \ref{prop5.1}, thanks to the elliptic estimates (\ref{eq47}), (\ref{eq48}), and as there, we find that
\begin{multline}
\label{eq51}
{\rm tr}\, \frac{1}{2\pi i} \int_{\gamma_1} \left(z-P_{\eps}\right)^{-1}\, \psi_{r,\pm}\,dz \\  =
\frac{1}{2\pi i} \frac{1}{(2\pi h)^2}\int_{\gamma_1} \int\!\!\!\int \frac{1}{z-p_{\eps}(\rho)}\, \psi_{r,\pm}(\rho)\, \mu(d\rho)\, dz +
{\cal O}\left(\frac{1}{h}\right).
\end{multline}

\medskip
\noindent
When understanding the trace
$$
{\rm tr}\,\frac{1}{2\pi i}\int_{\gamma_1} (z-P_{\eps})^{-1}\psi_{i,-}\,dz,
$$
we continue to follow the analysis of Section \ref{horizontal}, and introduce an auxiliary trace class perturbation of $P_{\eps}$, concentrated in a
region of the form $\abs{\Re P_{\eps}} < 1/{{\cal O}(1)}$, similar to (\ref{eq32.051}). The arguments of Section \ref{horizontal} apply then as the stand, and we get the following
direct analog of Proposition \ref{prop5.4}.

\begin{prop}
\label{prop6.3}
Assume that $E_2<0<E_4$, $\abs{E_j}< 1/{{\cal O}(1)}$, $j=2,4$, and let $0<\widetilde{\delta}<1$ be so small that
$h\leq \eps^{12\widetilde{\delta}}$. Let $\gamma_j$, $j=1,3$ be the horizontal segment given by $E_2 \leq \Re z \leq E_4$, $\Im z =\eps F_j$.
Let finally $0\leq \chi_j\in C^{\infty}_0(\Lambda)$ be a cut-off function to an $\eps^{\widetilde{\delta}}$-\neigh{} of $\widehat{\Lambda}_j$, $j=1,3$,
enjoying the commutator property {\rm (\ref{eq46})}. We have
\begin{multline}
{\rm tr}\,\frac{1}{2\pi i}\int_{\gamma_j} \left(z-P_{\eps}\right)^{-1}\,(1-\chi_j)\,dz \\ \nonumber
= \frac{1}{2\pi i}\frac{1}{(2\pi h)^2} \int_{\gamma_j} \int\!\!\!\int  \frac{1}{z-p_{\eps}(\rho)}(1-\chi_j(\rho))\,\mu(d\rho)\,dz
+ {\cal O}\left(\frac{\eps^{3\widetilde{\delta}}}{h^2}\right) \log \frac{1}{\eps},\quad j=1,3.
\end{multline}
Here we continue to assume that $h\ll \eps \leq {\cal O}(h^{\delta})$, $\delta>0$.
\end{prop}

\medskip
\noindent
An inspection, using the normal forms near the tori $\widehat{\Lambda}_j$, $j=1,3$, described in Proposition \ref{prop6.1}, shows next that the
result of Proposition \ref{prop5.5} remains valid in the present situation. Combining this observation with Propositions \ref{prop6.2} and \ref{prop6.3},
we conclude that the number of eigenvalues of $P_{\eps}$ in the rectangle $R$ in (\ref{eq44}) is given by
$$
\frac{1}{(2\pi h)^2}\int\!\!\!\int 1_R \left(p_{\eps}(\rho)\right)\, \mu(d\rho)
+ \frac{1}{h^2}{\cal O}\left(\eps^{3\widetilde{\delta}} \log{\frac{1}{\eps}}\right).
$$
Recalling the construction of the IR-manifold $\Lambda$, described in detail in \cite{HiSj04}, we may summarize the discussion in this section in the
following result, analogous to Theorem \ref{theo_main} in the periodic case.

\begin{theo}
\label{theo6.4}
Let $P_{\eps}$ satisfy the general assumptions of subsection {\rm 2.1}, in particular {\rm (\ref{eq9.2})}, and make fur\-ther the assum\-ption
{\rm (\ref{eq36})}. Let
$$
F_j \in [\min_{p^{-1}(0)\cap {\bf R}^4} \langle{q}\rangle, \max_{p^{-1}(0)\cap {\bf R}^4} \langle{q}\rangle],\quad j=1,3,\quad F_3 < F_1,
$$
be such that the assumptions {\rm (\ref{eq40})} and {\rm (\ref{eq41})} are satisfied. Furthermore, assume that the level sets
$\Lambda_{0,F_j}$ in {\rm (\ref{eq39})} are connected, $j=1,3$. Assume that $\eps  = {\cal O}(h^{\delta})$, $0<\delta\leq 1$, is such that $h/\eps \ll 1$.
Let $C>0$ be sufficiently large and let $E_2 < 0 < E_4$ be independent of $h$ with $\abs{E_j} < 1/C$, $j=2,4$. Assume finally that
$\widetilde{\delta}\in (0,1)$ is so small that $h\leq \eps^{12\widetilde{\delta}}$. Then the number of eigenvalues of $P_{\eps}$ in the rectangle
$$
R = (E_2,E_4) + i\eps (F_3,F_1),
$$
counted with the algebraic multiplicities, is equal to
$$
\frac{1}{(2\pi h)^2}\int\!\!\!\int_{E_2 \leq p \leq E_4} 1_{[F_3,F_1]}\left(\langle{q\rangle}\right)\, dx\,d\xi
+ \frac{1}{h^2}{\cal O}\left(\eps^{3\widetilde{\delta}} \log{\frac{1}{\eps}}\right).
$$
\end{theo}

\section{Complex perturbations and the damped wave equation on a surface of revolution}
\label{revolution}
\setcounter{equation}{0}
The purpose of this final section is to illustrate how Theorems \ref{theo_main} and \ref{theo_ext} apply in the case when $M$ is an analytic
surface of revolution in $\real^3$, and
\begeq
\label{eq60}
P_{\eps}=-h^2\Delta+i\eps q,
\endeq
where $\Delta$ is the Laplace-Beltrami operator and $q$ is an analytic function on $M$. When doing so, we shall restrict the attention to
the same class of surfaces of revolution as in~\cite{HiSjVu},~\cite{HiSj08}, and begin by recalling the assumptions made on $M$ in these previous works.

\medskip
\noindent
Let us normalize $M$ so that the $x_3$-axis is its axis of revolution, and parametrize it by the cylinder
$[0,L]\times S^1$, $L>0$,
\begeq
\label{eq61}
[0,L]\times S^1 \ni (s,\theta)\mapsto (f(s)\cos\theta, f(s)\sin \theta, g(s)).
\endeq
Here the parameter $s\in [0,L]$ is the arclength along the meridians, so that $(f'(s))^2+(g'(s))^2=1$. The functions $f$ and $g$ are assumed to be
real analytic on $[0,L]$, and we shall assume that for each $k\in \nat$,
$$
f^{(2k)}(0)=f^{(2k)}(L)=0,
$$
and that $f'(0)=1$, $f'(L)=-1$. As we recalled in~\cite{HiSjVu}, these assumptions guarantee the regularity of $M$ at the poles.
We assume furthermore that $M$ is a simple surface of revolution, in the sense that $f(s)>0$ on $(0,L)$ has precisely one critical point
$s_0\in (0,L)$, which is a non-degenerate maximum, so that $f''(s_0)<0$. To fix the ideas, assume that $f(s_0)=1$.
Associated to $s_0$ we have the equatorial geodesic $\subset M$, given by $s=s_0$, $\theta\in S^1$.

\medskip
\noindent
Writing
$$
T^*\left(M\backslash \{(0,0,g(0)),(0,0,g(L))\}\right)\simeq T^*\left((0,L)\times S^1\right),
$$
and using (\ref{eq61}), we see that the leading symbol of $P_0=-h^2\Delta$ on $M$ is given by
\begin{equation}
\label{eq62}
p(s,\theta,\sigma,\theta^*)=\sigma^2+\frac{(\theta^*)^2}{f^2(s)}.
\end{equation}
Here $\sigma$ and $\theta^*$ are the dual variables to $s$ and
$\theta$, respectively. Since the function $p$ in (\ref{eq62}) does not depend on $\theta$, it follows that $p$ and $\theta^*$ are in involution,
and we recover the well-known fact that the geodesic flow on $M$ is completely integrable.

\medskip
\noindent
Let $E>0$ and $\abs{F}< E^{1/2}$, $F\neq 0$. Then the set
$$
\Lambda_{E,F}: p=E,\,\, \theta^*=F,
$$
is an analytic Lagrangian torus contained inside the real energy surface $p^{-1}(E)$. Geometrically, the torus $\Lambda_{E,F}$ consists of
geodesics contained between and intersecting tangentially the parallels $s_{\pm}(E,F)$ on $M$ defined by the equation
$$
f(s_{\pm}(E,F))=\frac{\abs{F}}{E^{1/2}}.
$$
For $F=0$, the parallels reduce to the two poles and we obtain a torus consisting of a family of meridians. The case
$\abs{F}=E^{1/2}$ is degenerate and corresponds to the equator $s=s_0$, traversed with the two different orientations. Writing
$\Lambda_a:=\Lambda_{1,a}$, we get a decomposition as in (\ref{eq12}),
$$
p^{-1}(1)\cap T^*M = \bigcup_{a\in J} \Lambda_a,
$$
with $J = [-1,1]$, $S=\{\pm 1\}$.

\medskip
\noindent
In~\cite{HiSjVu}, we have derived an explicit expression for the
rotation number $\omega(\Lambda_a)$ of the torus $\Lambda_{a}\subset p^{-1}(1)$, $0\neq a\in (-1,1)$, given by
\begeq
\label{eq63}
\omega(\Lambda_a)=\frac{a}{\pi}\int_{s_-(a)}^{s_+(a)}
\frac{1}{f^2(s)}\left(1-\frac{a^2}{f^2(s)}\right)^{-1/2}\,ds,\quad
f(s_{\pm}(a))=\abs{a}.
\endeq
We are going to assume that the analytic function $(-1,1)\ni a\mapsto
\omega(\Lambda_a)$ is not identically constant.
%\color{red} H\"ar och
%l\"angre fram tycker jag att man kan ta bort detaljer om
%diofanticiteten och har gjort det i denna version (med \% tecken i k\"allfilen). \color{black}

%Expressed in terms of the
%rotation number $\omega(\Lambda_a)$, the Diophantine condition (\ref{eq19}%) takes the form
%\begeq
%\label{eq64}
%\abs{\omega(\Lambda_a)-\frac{p}{q}}\geq \frac{\alpha}{q^{2+d}},\quad p\in\%z,\,\, q\in \nat,\, q\neq 0,
%\endeq
%for some $\alpha>0$, $d>0$. When (\ref{eq64}) holds, we say that the torus% $\Lambda_a$ is $(\alpha,d)$--Diophantine.

\medskip
\noindent
Let $q=q(s,\theta)$ be a real-valued analytic function on $M$, which we shall view as
a function on $T^*M$. Associated to each $a\in J$, we introduce the
compact interval $Q_{\infty}(\Lambda_a)\subset \real$ defined as in (\ref{eq16}). We also define
an analytic function
$$
(-1,1)\ni a\mapsto \langle{q}\rangle(\Lambda_a),
$$
obtained by averaging $q$ over the invariant tori $\Lambda_a$. Let us assume that the function
$a\mapsto \langle{q}\rangle(\Lambda_a)$ is not identically constant.

\medskip
\noindent
{\it Example}. Assume that $q=q(s)$ depends on $s$ only. Then it was shown in~\cite{HiSjVu} that for all $0\neq a\in(-1,1)$, we have
$$
Q_{\infty}(\Lambda_a) = \{ \langle{q}\rangle(\Lambda_a)\},
$$
where
$$
\langle{q}\rangle(\Lambda_a) = \frac{J(q,a)}{J(1,a)},\quad J(\psi,a) = \int_{s_-(a)}^{s_+(a)} \psi(s) \frac{f(s)}{(f^2(s) - a^2)^{1/2}}\,ds,\quad
f(s_{\pm}(a))=\abs{a}.
$$

\medskip
\noindent
Generalizing the example above, an analytic family of the form
$$
q_{\eta}(s,\theta) = q_0(s) + \eta q_1(s,\theta),\quad q_0=q,\quad 0<\eta \ll 1,
$$
was also considered in~\cite{HiSjVu}, and it was shown that the set
$$
\bigcup_{\Lambda_a \in \omega^{-1}({\rm \bf Q})\cup S} Q_{\infty,\eta}(\Lambda_a),\quad S=\{\pm 1\},
$$
has a small measure compared with that of $\{\langle{q_{\eta}}\rangle(\Lambda_a); a\in (-1,1)\}$, and that consequently, there exists a
rich set of values $F \in \cup_{a\in J} Q_{\infty,\eta}(\Lambda_a)$ satisfying the assumptions (\ref{eq21}), (\ref{eq22}), (\ref{eq23}), and
(\ref{eq26.01}) for $q_{\eta}$. Here $Q_{\infty,\eta}(\Lambda_a)$ is the range of the limit of $\langle{q_{\eta}}\rangle_T$, as
$T\rightarrow \infty$, along $\Lambda_a$.

%{\color{red} H\"ar tar jag bort en bit ocks\aa{}.}
%\medskip
%\noindent
%Continuing to follow~\cite{HiSjVu}, let us now introduce the uniformly
%good values in $\real$, for which our results will be valid uniformly.
%Let $d>0$ be fixed. Given $\alpha$, $\beta$, $\gamma>0$, we say that
%$F \in \real$ is $(\alpha,\beta,\gamma)$--good if the following
%conditions hold:
%\begin{itemize}
%\item $F$ is not in the union of all $Q_{\infty}(\Lambda_a)$ with
%  ${\rm dist}(\Lambda_a,S)\leq \alpha$ or with $\omega(\Lambda_a)$ not
%  $(\alpha,d)$--Diophantine, in the sense of (\ref{eq64}).
%\item If $F = \langle{q}\rangle(\Lambda_a)$ then
%  $\abs{d_a\langle{q}\rangle(\Lambda_a)}\geq \alpha$, $\abs{d_a
%    \omega(\Lambda_a)}\geq \alpha$.
%\item Let $\langle{q}\rangle^{-1}(F)=\{ \Lambda_{a_1},\ldots\,
%  \Lambda_{a_L}\}$. Then the distance in $\real$ from $F$ to the union
%$$
%\bigcup_{\Lambda_a\in J;\, {\rm dist}_{J}\left(\Lambda_a,
%    \cup_{j=1}^{L} \Lambda_{a_j}\right) > \beta} Q_{\infty}(\Lambda_a)
%$$
%is $> \gamma$.
%\end{itemize}
%
%\medskip
%\noindent
\medskip
\noindent
The following result is a consequence of Theorem \ref{theo_main}.
\begin{prop}\label{rev1}
Assume that $M\subset \real^3$ is a simple analytic surface of revolution given by the parametrization {\rm (\ref{eq61})}, for
which the rotation number $\omega(\Lambda)$, $\Lambda=\Lambda_a$, defined in {\rm (\ref{eq63})}, is not identically constant. Consider an operator
of the form $P_{\eps}=-h^2\Delta+i\eps q$, with $q=q_{\eta}$ as above, such that the analytic function
$(-1,1)\ni a\mapsto \langle{q}\rangle(\Lambda)$,  given by the
torus averages of $q$, is not identically constant and extends continuously to $[-1,1]$. There exists a subset ${\cal E}_\eta \subset
\bigcup_{\Lambda \in J}Q_\infty (\Lambda )$ of measure $\mu(\eta )$ tending to zero with $\eta$,
% Let $\alpha$, $\beta$, $\gamma>0$. If $F_j$, $j=1,3$, $F_3 < F_1$,
% are $(\alpha,\beta,\gamma)$--good, then
such that the conclusion of Theorem {\rm \ref{theo_main}} holds uniformly for
$F_j\in \bigcup_{\Lambda \in J}Q_\infty (\Lambda )\setminus {\cal E}_{\eta}$, and gives the number of eigenvalues of $P_{\eps}$ in the region
$$
\left[E_2,E_4\right] + i\eps \left[F_3, F_1\right],\quad E_2 < 1 < E_4,\quad \abs{E_j-1}\sim \eps^{\widetilde{\delta}}.
$$
% and
%$$
%\left \{ E_2 \leq \Re z \leq E_4,\, F_3(\Re z) \leq \frac{\Im z}{\eps}
%  \leq F_1(\Re z)\right\},\quad E_2 < 0 < E_4,\quad \abs{E_j} \sim
%\frac{1}{{\cal O}(1)},
%$$
% respectively.
\end{prop}

\medskip
\noindent
{\it Remark}. A more precise description of the complement of the set ${\cal E}_{\eta}$ in Proposition 7.1,
given in terms of $(\alpha,\beta,\gamma)$--good values, can be found in Section 7 of~\cite{HiSjVu}.

\medskip
\noindent {\it Remark}. Let us remark finally that when discussing the
operator $P_{\eps}$ given by (\ref{eq60}), it would have been possible
to allow the analytic function $q$ on $M$ to depend holomorphically on
the spectral parameter $z\in {\rm neigh}(1,\comp)$, with $q$
real-valued for $z$ real --- see also the discussion in Section 6
of~\cite{HiSj05} for a similar observation in the periodic case. Such
an extension is motivated by the problem of studying spectral
asymptotics for the damped wave equation with an analytic damping
coefficient, considered on the analytic surface of revolution
$M$.

\medskip
\noindent
We shall now apply this remark to the situation of the damped wave
equation
$$
(-\Delta +2a(x)\partial _t+\partial _t^2)v(t,x)=0,\hbox{ on }{\bf
  R}\times M,
$$
where $a(x)$ is analytic and real-valued. An important role is played here by the stationary solutions $e^{i\tau t}u(x)$, $u\not\equiv 0$
and the corresponding eigenfrequencies $\tau \in {\bf C}$, given by
the equation
\begin{equation}\label{rev.1}
(-\Delta +2ia(x)\tau -\tau ^2)u(x)=0,
\end{equation}
and we are interested in the asymptotic distribution of the
eigenfrequencies $\tau $. We know that the large eigenfrequencies are
confined to a band along the real axis (see \cite{Le}, \cite{Sj00} for such results under
more general assumptions), and that the real parts obey the same Weyl law
as for the corresponding Helmholtz equation $(-\Delta -\tau
^2)u=0$. Less is known about the asymptotic distribution of the
imaginary parts, and here we can apply Proposition \ref{rev1} and the
subsequent remark.

\medskip
\noindent
The set of eigenfrequencies is symmetric under reflection in the
imaginary axis, so we can concentrate on the case when $\Re \tau \gg 1$. We make a semiclassical reduction by putting $\tau =w/h$, $0<h\ll
1$, $\Re w\sim 1$ and get
\begin{equation}\label{rev.2}
(-h^2\Delta +2ihwa(x)-w^2)u=0,
\end{equation}
and we can apply our results with $z=w^2$, $\eps=h$, $q=2\sqrt{z}a$, and with $P=-h^2\Delta $ as the unperturbed operator
with leading symbol $p(x,\xi )=\xi ^2$. It may now also be useful to recall the following general bounds on the imaginary part of an
eigenfrequency $w$ in (\ref{rev.2}),
$$
h\left(\lim_{T\rightarrow \infty}\inf_{p^{-1}(1)} \langle{a}\rangle_T -o(1)\right) \leq \Im w
\leq  h\left(\lim_{T\rightarrow \infty}\sup_{p^{-1}(1)} \langle{a}\rangle_T + o(1)\right).
$$
This result was obtained in~\cite{Le} --- see also (\ref{eq18}) for the present completely integrable case.

\medskip
\noindent
The function $\langle a\rangle_\infty = \lim_{T\rightarrow \infty} \langle{a}\rangle_T$ is
homogeneous of degree $0$ in $\xi$, thanks to the homogeneity properties of the $H_p$--flow. This means that the set $E_2\le \Re w\le E_4$,
$hF_3\le \Im w\le hF_1$ corresponds to the set
$\{ (x,\xi )\in T^*M;\, E_2^2\le \xi ^2\le E_4^2,\ F_3\le \langle
a\rangle_\infty \le F_1\}$ and we notice that the conditions imposed
on $F_j$ (i.e. on the
properties of the corresponding torus, where
$p=E$, $\langle a\rangle_{\infty}=F_j$) in Proposition \ref{rev1} in the
case when $q=a$, are independent of the real energy
$p=E$. Applying Proposition \ref{rev1} and the subsequent remark, we get
\begin{theo}
\label{rev2}
Consider the stationary damped wave equation in the equivalent forms
{\rm (\ref{rev.1})} and {\rm (\ref{rev.2})}. Assume that the assumptions of
Proposition {\rm \ref{rev1}} are fulfilled with $q=a$. Then uniformly for
$F_1,F_3\in \bigcup_{\Lambda \in J}Q_\infty (\Lambda ) \setminus {\cal E}_\eta $,
the number of eigenfrequencies $w$ of {\rm (\ref{rev.2})} in
the region
$$
[E_2,E_4]+ih[F_3,F_1],\ E_2<1<E_4,\ |E_j-1|\sim  h^{\widetilde{\delta }},
$$
is equal to $(2\pi h)^{-2}$ times
\begin{equation*}
\mathrm{vol\,}\{(x,\xi )\in T^*M;\, E_2^2\le
\xi ^2\le E_4^2,\ Q_\infty (\Lambda \left(x,\frac{\xi }{|\xi
  |}\right))\subset [F_3,F_1]\}+{\cal O}(h^{3\widetilde{\delta}}\ln \frac{1}{h}).
\end{equation*}

\par
\noindent
More generally, the number of eigenfrequencies $w$ of {\rm (\ref{rev.2})} in the region
$$
[E_2,E_4]+ih[F_3,F_1],\ E_2<1<E_4,\ |E_j|\sim 1,
$$
is equal to
\begin{equation*}
\begin{split}
\frac{1}{(2\pi h)^2}\left(\mathrm{vol\,}\{(x,\xi )\in T^*M;\, E_2^2\le
\xi ^2\le E_4^2,\ Q_\infty (\Lambda (x,\frac{\xi }{|\xi
  |}))\subset [F_3,F_1]\}\right)+{\cal O}(h^{2\widetilde{\delta
}-2}).
\end{split}
\end{equation*}

\par
\noindent
The number of eigenfrequencies $\tau $ of {\rm (\ref{rev.1})} in the
region $[E_2,E_4]+i[F_3,F_1]$, where $E_2\le E_4$, $E_j\sim 1/h\gg
1$, is equal to
\begin{equation*}
\frac{1}{(2\pi)^2}\left(\mathrm{vol\,}\{(x,\xi )\in T^*M;\, E_2^2\le
\xi ^2\le E_4^2,\ Q_\infty (\Lambda \left(x,\frac{\xi }{|\xi
  |}\right))\subset [F_3,F_1]\}\right)+{\cal O}(h^{2\widetilde{\delta
}-2}).
\end{equation*}

\end{theo}


\begin{thebibliography}{40}

\bibitem{A} N. Anantharaman, {\it Spec\-tral de\-via\-ti\-ons for the dam\-ped wa\-ve equa\-tion}, GAFA {\bf 20} (2010), 593--626.

\bibitem{BF} A. V. Bolsinov and A. T. Fomenko, {\it Integrable Hamiltonian systems. Geometry, topology, classification},
Chapman \& Hall/CRC, Boca Raton, FL, 2004.

\bibitem{Davies} E. B. Davies, {\it Pseudospectra, the harmonic oscillator and complex resonances},
R. Soc. Lond. Proc. Ser. A Math. Phys. Eng. Sci. {\bf 455}  (1999), 585--599.

\bibitem{D} J. J. Duistermaat, {\it On global action-angle coordinates}, Comm. Pure Appl. Math. {\bf 33} (1980), 687--706.

%\bibitem{BBGS} G. Bennetin, L. Galgani, A. Giorgilli, and J.-M. Strelcyn, {\it A proof of Kolmogorov's theorem on
%invariant tori using canonical transformations defined by the Lie method}, Nuovo Cimento B {\bf 79} (1984), 201--223.

\bibitem{Colin} Y. Colin de V\`erdiere, {\it La m\'ethode de moyennisation en m\'ecanique semi-classique}, Journ\'ees "\'Equations aux D\'eriv\'ees
Partielles", Saint-Jean-de-Monts, 1996.

%\bibitem{DeSjZw} N. Dencker, J. Sj\"ostrand, and M. Zworski, {\it Pseudo-\-spectra of se\-mi\-clas\-si\-cal
%(pseudo)\-diffe\-ren\-tial ope\-rators}, Comm. Pure Appl. Math. {\bf 57} (2004), 384-415.

\bibitem{DiSj} M. Dimassi and J. Sj\"ostrand, {\it Spectral asymptotics in the semi-classical limit}, Cambridge
University Press, 1999.

\bibitem{GoKr} I. C. Gohberg and M. G. Krein, {\it Introduction to the theory of linear nonselfadjoint operators},
Translations of Mathematical Monographs, Vol. 18 American Mathematical Society, Providence, R.I. 1969.

\bibitem{Guillemin81} V. Guillemin, {\it Band asymptotics in two dimensions},
Adv. in Math. {\bf 42} (1981), 248--282.

\bibitem{GuSt} V. Guillemin and M. Stenzel, {\it Grauert tubes and the homogeneous Monge-Amp\`ere equations},
J. Differential Geom. {\bf 34} (1991), 561--570.

\bibitem{HSj} M. Hager and J. Sj\"ostrand, {\it Eigenvalue asymptotics for randomly perturbed non-selfadjoint operators},
Math. Ann. {\bf 342} (2008), 177-?243.

%\bibitem{Hi02} M. Hitrik, {\it Eigenfrequencies for damped wave equations on Zoll manifolds},
%Asymptot. Analysis {\bf 31} (2002), 265--277.

%\bibitem{Hi03} M. Hitrik, {\it Eigenfrequencies and expansions for damped wave equations}, Methods Appl. Anal.
%{\bf 10} (2003), 543--564.

\bibitem{Hi1} M. Hitrik, {\it Boundary spectral behavior for semiclassical operators in dimension one},
Int. Math. Res. Not. {\bf 64} (2004), 3417--3438.

\bibitem{HiSj04} M. Hitrik and J. Sj\"ostrand, {\it Non-selfadjoint perturbations of selfadjoint operators in {\rm 2} dimensions {\rm I}},
Ann. Henri Poincar\'e {\bf 5} (2004), 1--73.

\bibitem{HiSj05} M. Hitrik and J. Sj\"ostrand, {\it Non-selfadjoint perturbations of selfadjoint operators in {\rm 2}
dimensions {\rm II}. Vanishing averages}, Comm. Partial Differential Equations {\bf 30} (2005), 1065--1106.

\bibitem{HiSj3a} M. Hitrik and J. Sj\"ostrand, {\it Non-selfadjoint perturbations of selfadjoint operators in
{\rm 2} dimensions {\rm III a}. One branching point}, Canadian J. Math. {\bf 60} (2008), 572--657.

\bibitem{HiSj08} M. Hitrik and J. Sj\"ostrand, {\it Rational invariant tori, phase space tunneling, and spectra for
non-selfadjoint operators in dimension {\rm 2}}, Ann. Sci. E.N.S. {\bf 41} (2008), 511--571.

\bibitem{HiSjVu} M. Hitrik, J. Sj\"ostrand, and S. V\~u Ng\d{o}c, {\it
Diophantine tori and spectral asymptotics for non-selfadjoint operators},
Amer. J. Math. {\bf 129} (2007), 105--182.

\bibitem{Ho} L. H\"ormander, {\it The analysis of linear partial differential operators {\rm I}}, Springer-Verlag,
Berlin, 1983.

\bibitem{Ivrii} V. Ivrii, {\it Microlocal analysis and precise spectral asymptotics}, Springer Verlag, 1998.

\bibitem{Lazutkin} V. F. Lazutkin, {\it KAM theory and semiclassical approximations to eigenfunctions}. With an Addendum by A. I. Shnirelman. Ergebnisse
der Mathematik und Ihrer Grenzgebiete, Springer Verlag, Berlin, 1993.

\bibitem{Le} G. Lebeau, {\it \'Equation des ondes amorties}, Algebraic and geometric methods in mathematical physics
(Kaciveli, 1993), 73--109, Math. Phys. Stud., {\bf 19}, Kluwer Acad. Publ., Dordrecht, 1996.

\bibitem{Lerner} N. Lerner, {\it Metrics on the phase space and non-selfadjoint pseudo-differential operators},
Pseudo-Differential Operators. Theory and Applications, 3. Birkh\"auser Verlag, Basel, 2010

\bibitem{Ma} A. S. Markus, {\it Introduction to the spectral theory of polynomial operator pencils}, Translations
of Mathematical Monographs 71, American Mathematical Society, Providence RI, 1998.

\bibitem{MaMa} A. S. Markus and V. I. Matsaev, {\it Comparison theorems for spectra of linear operators and spectral
asymptotics} (Russian), Trudy Moskov. Mat. Obsch. {\bf 45} (1982), 133--181.

\bibitem{MeSj1} A. Melin and J. Sj\"ostrand,
{\it Determinats of pseudodifferential operators and complex deformations of phase space}, Methods and Applications
of Analysis {\bf 9} (2002), 177--238.

\bibitem{MeSj2} A. Melin and J. Sj\"ostrand, {\it Bohr-Sommerfeld quantization condition for non-selfadjoint operators
in dimension {\rm 2}}, Ast\'erisque {\bf 284} (2003), 181--244.

\bibitem{PeZw} V. Petkov and M. Zworski, {\it Semi-classical estimates on the scattering determinant}, Ann. Henri Poincar\'e {\bf 2}  (2001), 675--71.

\bibitem{Shnirelman} A. Shnirelman, {\it On the asymptotic properties of the eigenfunctions in regions of chaotic motion},
Addendum to~\cite{Lazutkin}.

\bibitem{Sj74} J. Sj\"ostrand, {\it Parametrices for pseudodifferential operators with multiple characteristics},
Ark. Mat. {\bf 12} (1974), 85--130.

\bibitem{Sj82} J. Sj\"ostrand, {\it Singularit\'es analytiques microlocales}, Ast\'erisque, 1982.

%\bibitem{Sj90} J. Sj\"ostrand, {\it Geometric bounds on the density of
%resonances for semiclassical problems}, Duke Math. Journal {\bf 60} (1990), 1--57.

\bibitem{Sj95} J. Sj\"ostrand, {\it Function spaces associated to global I-Lagrangian manifolds},
Structure of solutions of differential equations, Katata/Kyoto, 1995, World Sci. Publ., River Edge, NJ (1996).

\bibitem{Sj96} J.~Sj\"ostrand, {\it Density of resonances for strictly convex analytic obstacles}, Can. J. Math.
{\bf 48} (1996), 397--447.

\bibitem{Sj97} J. Sj\"ostrand, {\it A trace formula and review of some estimates for resonances}, Microlocal analysis
and spectral theory (Lucca, 1996), 377--437, NATO Adv. Sci. Inst. Ser. C Math. Phys. Sci., {\bf 490}, Kluwer Acad.
Publ., Dordrecht, 1997.

\bibitem{Sj00} J. Sj\"ostrand, {\it Asymptotic distribution of eigenfrequencies for damped wave equations}, Publ. Res.
Inst. Math. Sci. {\bf 36} (2000), 573--611.

\bibitem{Sj01} J. Sj\"ostrand, {\it Resonances for bottles and trace formulae}, Math. Nachr. {\bf 221} (2001), 95--149.

\bibitem{Sj08} J. Sj\"ostrand, {\it Eigenvalue distributions and Weyl laws for semi-classical non-self-adjoint operators in {\rm 2} dimensions},
to appear in volume in honor of J. J. Duistermaat, Progress in Mathematics, Birkh\"auser.

%\bibitem{Sj2003} J. Sj\"ostrand, {\it Perturbations of selfadjoint operators with periodic classical flow}, RIMS
%Kokyuroku 1315 (April 2003), "Wave phenomena and asymptotic analysis", 1--23.

\bibitem{SjZw1} J. Sj\"ostrand and M. Zworski, {\it Asymptotic distribution of resonances for convex obstacles}, Acta
Math. {\bf 183} (1999), 191--253.

\bibitem{Viola} J. Viola, {\it Resolvent estimates for non-selfadjoint operators with double characteristics}, {\sf http://arxiv.org/abs/0910.2511}.

\bibitem{Vodev} G. Vodev, {\it Sharp bounds on the number of scattering poles for perturbations of the Laplacian},
Comm. Math. Phys. {\bf 146} (1992), 205--216.

\bibitem{San} S. V\~u Ng\d{o}c, {\it Syst\`emes int\'egrables
semi-classiques: du local au global}, Panoramas et Synth\`eses 22, Soci\'et\'e Math\'ematique de France, Paris, 2006.

\bibitem{Weinstein77} A. Weinstein, {\it Asymptotics of eigenvalue clusters for the Laplacian plus a potential},
Duke Math. J. {\bf 44} (1977), 883--892.

\bibitem{Weyl} H. Weyl, {\it Das asymptotische Verteilungsgesetz der Eigenwerte linearer partieller
Differentialgleichungen (mit einer Anwendung auf die Theorie der Hohlraumstrahlung)},
Math. Ann. {\bf 71} (1912), 441--479.

\bibitem{Zw} M. Zworski, {\it Sharp polynomial bounds on the number of scattering poles}, Duke Math. J. {\bf 59} (1989), 311--323.

\end{thebibliography}
\end{document}